\documentclass[12pt]{article}
\usepackage[dvips]{graphicx}
\usepackage{pp,psfrag,rotating,hyperref}

\newcommand{\cL}{\mathcal{L}}

\newcommand{\hZ}{\hat{Z}}

\newcommand{\fS}{\mathfrak{S}}
\newcommand{\fP}{\mathfrak{P}}
\newcommand{\Tr}{\mathrm{Tr}}

\title{Proof of the Labastida-Mari\~{n}o-Ooguri-Vafa Conjecture}
\author{Kefeng Liu and Pan Peng}
\date{}

\newtheorem*{lmovconj}{Conjecture (LMOV)}

\DeclareMathOperator{\lk}{lk}
\newcommand{\chR}{\check{\mathcal{R}}}
\newcommand{\fsl}{\mathfrak{sl}}
\newcommand{\fu}{\mathfrak{u}}
\newcommand{\Mbar}{\overline{M}}

\def\mathcenter#1{%
  \vcenter{\hbox{#1}}%
}

\newdimen\tableauside\tableauside=1.0ex
\newdimen\tableaurule\tableaurule=0.4pt
\newdimen\tableaustep
\def\phantomhrule#1{\hbox{\vbox to0pt{\hrule height\tableaurule width#1\vss}}}
\def\phantomvrule#1{\vbox{\hbox to0pt{\vrule width\tableaurule height#1\hss}}}
\def\sqr{\vbox{%
  \phantomhrule\tableaustep
  \hbox{\phantomvrule\tableaustep\kern\tableaustep\phantomvrule\tableaustep}%
  \hbox{\vbox{\phantomhrule\tableauside}\kern-\tableaurule}}}
\def\squares#1{\hbox{\count0=#1\noindent\loop\sqr
  \advance\count0 by-1 \ifnum\count0>0\repeat}}
\def\partition#1{\vcenter{\offinterlineskip
  \tableaustep=\tableauside\advance\tableaustep by-\tableaurule
  \kern\normallineskip\hbox
    {\kern\normallineskip\vbox
      {\gettableau#1 0 }%
     \kern\normallineskip\kern\tableaurule}%
  \kern\normallineskip\kern\tableaurule}}
\def\gettableau#1 {\ifnum#1=0\let\next=\null\else
  \squares{#1}\let\next=\gettableau\fi\next}

\begin{document}
\maketitle

\begin{abstract}
Based on large $N$ Chern-Simons/topological string duality, in a
series of papers \cite{OV,LMV,LM}, J.M.F.~Labastida, M.~Mari\~{n}o,
H.~Ooguri and C.~Vafa conjectured certain remarkable new algebraic
structure of link invariants and the existence of infinite series of
new integer invariants. In this paper, we provide a proof of this
conjecture. Moreover, we also show these new integer invariants
vanish at large genera.
\end{abstract}

\tableofcontents

\section{Introduction}

\subsection{Overview}

For decades, we have witnessed the great development of string theory and its
powerful impact on the development of mathematics. There have been a lot of
marvelous results revealed by string theory, which deeply relate different
aspects of mathematics. All these mysterious relations are connected by a core
idea in string theory called ``duality". It was found that string theory on
Calabi-Yau manifolds provided new insight in geometry of these spaces. The
existence of a topological sector of string theory leads to a simplified model
in string theory, the topological string theory.

A major problem in topological string theory is how to compute Gromov-Witten
invariants. There are two major methods widely used: mirror symmetry in
physics and localization in mathematics. Both methods are effective
when genus is low while having trouble in dealing with higher genera due to
the rapidly growing complexity during computation. However, when the
target manifold is Calabi-Yau threefold, large $N$ Chern-Simons/topological
string duality opens a new gate to a complete solution of computing
Gromov-Witten invariants at all genera.

The study of large $N$ Chern-Simons/topological string duality was originated in
physics by an idea that gauge theory should have a string theory explanation. In
1992, Witten \cite{W2} related topological string theory of $T^\ast M$ of a
three dimensional manifold $M$ to Chern-Simons gauge theory on $M$. In 1998,
Gopakumar and Vafa \cite{GV} conjectured that, at large $N$, open topological
A-model of $N$ D-branes on $T^{\ast }S^{3}$ is dual to closed topological string
theory on resolved conifold $\mathcal{O}(-1)\oplus\mathcal{O}(-1)\rightarrow
\mathbb{P}^1$. Later, Ooguri and Vafa \cite{OV} showed a picture on how to
describe Chern-Simons invariants of a knot by open topological string theory on
resolved conifold paired with lagrangian associated with the knot.

Though large $N$ Chern-Simons/topological string duality still
remains open, there have been a lot of progress in this direction
demonstrating the power of this idea. Even for the simplest knot,
the unknot, Mari\~{n}o-Vafa formula \cite{MV, LLZ1} gives a
beautiful closed formula for Hodge integral up to three Chern
classes of Hodge bundle. Furthermore, using topological vertex
theory \cite{AKMV, LLZ2, LLLZ}, one is able to compute Gromov-Witten
invariants of any toric Calabi-Yau threefold by reducing the
computation to a gluing algorithm of topological vertex. This thus
leads to a closed formula of topological string partition function,
a generating function of Gromov-Witten invariants, in all genera for
any toric Calabi-Yau threefolds.

On the other hand, after Jones' famous work on polynomial knot
invariants, there had been a series of polynomial invariants
discovered (for example, \cite{J, HOMFLY, K}), the generalization of
which was provided by quantum group theory \cite{T} in mathematics
and by Chern-Simons path integral with the gauge group $SU(N)$
\cite{W1} in physics.

Based on the large $N$ Chern-Simons/topological string duality, Ooguri and Vafa
\cite{OV} reformulated knot invariants in terms of new integral invariants
capturing the spectrum of M2 branes ending on M5 branes embedded in the resolved
conifold. Later, Labastida, Mari\~{n}o and Vafa \cite{LMV,LM} refined the
analysis of \cite{OV} and conjectured the precise integrality structure for open
Gromov-Witten invariants. This conjecture predicts a remarkable new algebraic
structure for the generating series of general link invariants and the
integrality of infinite family of new topological invariants. In string theory,
this is a striking example that two important physical theories, topological
string theory and Chern-Simons theory, exactly agree up to all orders. In
mathematics this conjecture has interesting applications in understanding the
basic structure of link invariants and three manifold invariants, as well as the
integrality structure of open Gromov-Witten invariants. Recently, X.S. Lin and
H. Zheng \cite{LZ} verified LMOV conjecture in several lower degree cases for
some torus links.

In this paper, we give a complete proof of Labastida-Mari\~{n}o-Ooguri-Vafa
conjecture for any link (We will briefly call it LMOV conjecture).
First, let us describe the conjecture and the main ideas of the proof. The
details can be found in Sections \ref{sec: existence} and
\ref{sec: integrality}.

\subsection{Labastida-Mari\~{n}o-Ooguri-Vafa conjecture}
\label{subsec: LMOV conjecture}

Let $\mathcal{L}$ be a link with $L$ components and $\mathcal{P}$ be the set of
all partitions. The Chern-Simons partition function of $\mathcal{L}$ is given by
\begin{align}\label{eqn: definition of Chern-Simons partition function in sec0}
 Z_{\textrm{CS}}(\mathcal{L};\, q,t)
=\sum_{\vec{A}\in \mathcal{P}^L}W_{\vec{A}}(\mathcal{L};\, q,t)
    \prod_{\alpha=1}^L s_{A^\alpha}(x^\alpha)
\end{align}
for any arbitrarily chosen sequence of variables
\[
 x^\alpha =(x^\alpha_1,x^\alpha_2,\ldots,)\,.
\]
In \eqref{eqn: definition of Chern-Simons partition function in sec0},
$W_{\vec{A}}(\mathcal{L})$ is the quantum group invariants of $\mathcal{L}$
labeled by a sequence of partitions $\vec{A}=(A^1,\ldots,A^L)\in \mathcal{P}^L$
which correspond to the irreducible representations of quantized universal
enveloping algebra $U_q(\fsl(N,\mathbb{C}))$, $s_{A^\alpha}(x^\alpha)$ is the
Schur function.

Free energy is defined to be
\[
 F=\log Z_{\textrm{CS}}\,.
\]
Use plethystic exponential, one can obtain
\begin{equation}\label{eqn: formula of f_A}
  F
= \sum_{d=1}^{\infty }\sum_{\vec{A}\neq 0}
    \frac{1}{d}f_{\vec{A}}(q^{d},t^{d})
    \prod_{\alpha =1}^{L}
    s_{A^{\alpha }}\big( ( x^\alpha)^{d} \big)\,,
\end{equation}
where
\[
 ( x^{\alpha }) ^{d}
= \big((x_{1}^\alpha)^d, (x_{2}^\alpha)^d,\ldots \big) \,.
\]

Based on the duality between Chern-Simons gauge theory and topological string
theory, Labastida, Mari\~{n}o, Ooguri, Vafa conjectured that $f_{\vec{A}}$ have
the following highly nontrivial structures.

For any $A$, $B\in \mathcal{P}$, define the following function
\begin{equation}\label{eqn: formula of M_AB}
 M_{AB}(q)
=\sum_{\mu} \frac{\chi_A(C_\mu) \chi_B(C_\mu)}{\mathfrak{z}_\mu}
 \prod_{j=1}^{\ell(\mu)} (q^{-\mu_j/2}-q^{\mu_j/2})\, .
\end{equation}

\begin{lmovconj}\label{conj: LMOV}
For any $\vec{A}\in \mathcal{P}^{L}$,
\begin{itemize}
\item[$(\mathrm{i})$.] there exists $P_{\vec{B}}(q,t)$ for
$\forall\vec{B}\in\mathcal{P}^L$, such that
\begin{equation}\label{eqn: f_A of P_B}
 f_{\vec{A}}(q,t)
=\sum_{ |B^\alpha|=|A^\alpha|} P_{\vec{B}}(q,t)
 \prod_{\alpha =1}^L M_{A^\alpha B^\alpha }(q).
\end{equation}
Furthermore, $P_{\vec{B}}(q,t)$ has the following expansion
\begin{equation}\label{eqn: P_B}
 P_{\vec{B}}(q,t)
=\sum_{g= 0}^\infty \sum_{Q\in \mathbb{Z}/2} N_{\vec{B};\,g,Q}
 (q^{-1/2}-q^{1/2})^{2g-2}t^{Q}\,.
\end{equation}

\item[$(\mathrm{ii})$.] $N_{\vec{B};\,g,Q}$ are integers.
\end{itemize}
\end{lmovconj}

For the meaning of notations, the definition of quantum group invariants of
links and more details, please refer to Section \ref{sec: lmov}.

This conjecture contains two parts:
\begin{itemize}
 \item The existence of the special algebraic structure
\eqref{eqn: P_B}.

 \item The integrality of the new invariants $N_{\vec{B};\, g,Q}$.
\end{itemize}

If one looks at the right hand side of \eqref{eqn: P_B}, one will find it very
interesting that the pole of $f_{\vec{A}}$ in $(q^{-1/2}-q^{1/2})$ is actually
at most of order $1$ for any link and any labeling partitions.
However, by the calculation of quantum group invariants of links, the pole order
of $f_{\vec{A}}$ might be going to $\infty$ when the degrees of labeling
partitions go higher and higher. This miracle cancelation
implies a very special algebraic structure of quantum group invariants of links
and thus the Chern-Simons partition function. in Section \ref{subsec:
application to knot theory}, we include an example in the simplest setting
showing that this cancelation has shed new light on the quantum group
invariants of links.

\subsection{Main ideas of the proof}

In our proof of LMOV conjecture, there are three new techniques.
\begin{itemize}
 \item
When dealing with the existence of the conjectured algebraic structure, one
will encounter the problem of how to control the pole order of
$(q^{-1/2}-q^{1/2})$. We consider the \emph{framed partition function}
$Z(\mathcal{L};\,q,t,\tau )$ of Chern-Simons invariants of links which
satisfies the following cut-and-join equation
\begin{align*}
&   \frac{\partial Z(\mathcal{L};\,q,t,\tau )}{\partial \tau }
\\
& = \frac{u}{2} \sum_{\alpha =1}^{L}
   \sum_{i,j\geq 1}
    \bigg(
     ijp^\alpha_{i+j}
     \frac{\partial ^{2}}
    {\partial p^\alpha_{i}\partial p^\alpha_{j}}
    +(i+j)p^\alpha_{i}p^\alpha_{j}
     \frac{\partial }{\partial p^\alpha_{i+j}}
    \bigg)
    Z(\mathcal{L};\,q,t,\tau ) \,.
\end{align*}
Here $p^\alpha_n=p_n(x^\alpha)$ are regarded as independent variables.

However, a deeper understanding of this conjecture relies on the following log
cut-and-join equation
\begin{align*}
 & \frac{\partial F(\mathcal{L};\,q,t,\tau )}{\partial \tau }
\\
 &=\frac{u}{2}\sum_{\alpha =1}^{L}
    \sum_{i,j\geq 1}
    \bigg(
      ijp^\alpha_{i+j}\frac{\partial^{2}F}
    {\partial p^\alpha_i\partial p^\alpha_j}
             +(i+j)p^\alpha_i p^\alpha_j
        \frac{\partial F}{\partial p^\alpha_{i+j}}
            +ijp^\alpha_{i+j}
      \frac{\partial F}{\partial p^\alpha_i}
        \frac{\partial F}{\partial p^\alpha_j}
    \bigg)\,.
\end{align*}
This observation is based on the duality of Chern-Simons theory and open
Gromov-Witten theory. The log cut-and-join equation is a non-linear ODE systems
and the non-linear part reflects the essential recursion structure of
Chern-Simons partition function. The miracle cancelation of lower order terms of
$q^{-1/2}-q^{1/2}$ occurring in free energy can be indicated in the formulation
of generating series of open Gromov-Witten invariants on the geometric side.

A powerful tool to control the pole order of $(q^{-1/2}-q^{1/2})$ through log
cut-and-join equation is developed in this paper as we called \emph{cut-and-join
analysis}. An important feature of cut-and-join equation shows that the
differential equation at
partition\footnote{Here, we only take a knot as an example. For the
case of links, it is then a natural extension.}
$(d)$ can only have terms obtained from
joining\footnote{Joining means that a partition is obtained from combining
two rows of a Young diagram and reforming it into a new partition, while cutting
means that a partition is obtained by cutting a row of a Young diagram into two
rows and reforming it into a new partition.}
while at $(1^d)$, non-linear terms vanishes and there is no joining terms. This
special feature combined with the degree analysis will squeeze out the desired
degree of $(q^{-1/2}-q^{1/2})$.

 \item
We found a rational function ring which characterizes the algebraic structure
of Chern-Simons partition function and hence (open) topological string partition
function by duality.

A similar ring characterizing closed topological string partition function
firstly appears in the second author's work on Gopakumar-Vafa conjecture
\cite{P1,P2}. The original observation in the closed case comes from the
structure of $R$-matrix in quantum group theory and gluing algorithm in the
topological vertex theory.

However, the integrality in the case of open topological string theory is
more subtle than the integrality of Gopakumar-Vafa invariants in the closed
case. This is due to the fact that the reformulation of Gromov-Witten
invariants as Gopakumar-Vafa invariants in the closed case weighted by power of
curve classes, while in the open case, the generating function is weighted by
the labeling irreducible representation of $U_q(\fsl(N,\mathbb{C}))$. This
subtlety had already been explained in \cite{LMV}.

To overcome this subtlety, one observation is that quantum group invariants look
Schur-function-like. This had already been demonstrated in the topological
vertex theory (also see \cite{Lukac}).
We refine the ring in the closed case and get the new ring
$\mathfrak{R}(y;\,q,t)$ (cf. section \ref{subsec: Integrality - ring
structure}). Correspondingly, we consider a new generating series of
$N_{\vec{B};\,g,Q}$, $T_{\vec{d}}$, as defined in \eqref{eqn: def of T_d}.

 \item
To prove $T_{\vec{d}}\in\mathfrak{R}(y;\,q,t)$, we combine with the multi-cover
contribution and $p$-adic argument therein. Once we can prove that $T_{\vec{d}}$
lies in $\mathfrak{R}$, due to the pole structure of the ring $\mathfrak{R}$,
the vanishing of $N_{\vec{B};\, g,Q}$ has to occur at large genera, we actually
proved
\[
 \sum_{g=0}^\infty \sum_{Q\in \mathbb{Z}/2} N_{\vec{B};\, g,Q}
 (q^{-1/2}-q^{1/2})^{2g} t^Q
\in \mathbb{Z} [(q^{-\frac{1}{2}}-q^{\frac{1}{2}})^2, t^{\pm \frac{1}{2}}]\,.
\]
\end{itemize}

The paper is organized as follows. In Section \ref{sec: partition}, we introduce
some basic notations about partition and generalize this concept to simplify our
calculation in the following sections. Quantum group invariants of links and
main results are introduced in Section \ref{sec: lmov}. In Section \ref{sec:
Hecke algebra}, we review some knowledge of Hecke algebra used in this paper. In
Section \ref{sec: existence} and \ref{sec: integrality}, we give the proof of
Theorem \ref{thm: existence} and \ref{thm: integrality} which answer LMOV
conjecture. In the last section, we discuss some problems related to LMOV
conjecture for our future research.

\subsection{Acknowledgments}

The authors would like to thank Professors S.-T. Yau, F. Li and Z.-P. Xin for
valuable discussions. We would also want to thank Professor M. Mari\~no for
pointing out some misleading parts. K.~L. is supported by NSF grant. P.~P. is
supported by NSF grant DMS-0354737 and Harvard University.

Before he passed away, Professor Xiao-Song Lin had been very interested in the
LMOV conjecture, and had been working on it. We would like to dedicate this
paper to his memory.

\section{Preliminary}\label{sec: partition}

\subsection{Partition and symmetric function}

A \emph{partition} $\lambda $ is a finite sequence of positive integers
$(\lambda_1,\lambda_2,\cdots)$
such that
\[
 \lambda_1\geq \lambda_2\geq \cdots \,.
\]
The total number of parts in $\lambda $ is called the length of $\lambda $ and
denoted by $\ell(\lambda )$. We use $m_i(\lambda)$ to denote the number of times
that $i$ occurs in $\lambda$. The degree of $\lambda$ is defined to be
\[
 |\lambda|=\sum_i\lambda_i\,.
\]
If $|\lambda |=d$, we say $\lambda $ is a partition of $d$. We also use notation
$\lambda \vdash d$. The automorphism group of $\lambda $, $\Aut\lambda$,
contains all the permutations that permute parts of $\lambda $ while still
keeping it as a partition. Obviously, the order of $\Aut\lambda $ is given by
\[
 |\Aut\lambda |=\prod_{i}m_{i}(\lambda )!\,.
\]
There is another way to rewrite a partition $\lambda $ in the following format
\[
 (1^{m_{1}(\lambda )}2^{m_{2}(\lambda )}\cdots)\,.
\]

A traditional way to visualize a partition is to identify a partition as a
\emph{Young diagram}. The Young diagram of $\lambda $ is a 2-dimensional graph
with
$\lambda _{j}$ boxes on the $j$-th row, $j=1,2,...,\ell (\lambda )$. All the
boxes are put to fit the left-top corner of a rectangle. For example
\[
 (5,4,2,2,1)=(12^{2}45)=\partition{5 4 2 2 1}.
\]
For a given partition $\lambda $, denote by $\lambda^t$ the conjugate
partition of $\lambda $. The Young diagram of $\lambda^t$ is transpose to the
Young diagram of $\lambda $: the number of boxes on $j$-th column of
$\lambda^t$ equals to the number of boxes on $j$-th row of $\lambda$, where
$1\leq j\leq \ell(\lambda )$.

By convention, we regard a Young diagram with no box as the partition of $0$ and
use notation $(0)$. Denote by $\mathcal{P}$ the set of all partitions. We can
define an operation $``\cup "$ on $\mathcal{P}$. Given two partitions $\lambda $
and $\mu $, $\lambda \cup \mu $ is the partition by putting all the parts of
$\lambda $ and $\mu $ together to form a new partition. For example
\[
 (12^{2}3)\cup (15) = (1^{2}2^{2}35).
\]
Using Young diagram, it looks like
\[
 \partition{3 2 2 1}\quad \cup \quad \partition{5 1}\
= \ \partition{5 3 2 2 1 1}\,.
\]

The following number associated with a partition $\lambda $ is used
throughout this paper,
\begin{align*}
 \mathfrak{z}_{\lambda }
=\prod_j j^{m_j(\lambda )}m_j(\lambda )!\,, \qquad
 \kappa_\lambda =\sum_j\lambda_j(\lambda_j -2j+1)\,.
\end{align*}
It's easy to see that
\begin{equation}\label{eqn: kappa_lambda symmetry}
 \kappa _{\lambda } = -\kappa _{\lambda ^{t}}.
\end{equation}

A power symmetric function of a sequence of variables $x=(x_{1},x_{2},...)$
is defined as follows
\[
 p_{n}(x)=\sum_{i}x_{i}^{n}.
\]
For a partition $\lambda$,
\[
 p_{\lambda }(x) = \prod_{j=1}^{\ell (\lambda )}p_{\lambda _{j}}(x).
\]

It is well-known that every irreducible representation of symmetric group can be
labeled by a partition. Let $\chi_\lambda$ be the character of the irreducible
representation corresponding to $\lambda$. Each conjugate class of symmetric
group can also be represented by a partition $\mu$ such that the permutation in
the conjugate class has cycles of length $\mu_1,\ldots, \mu_{\ell(\mu)}$.
Schur function $s_{\lambda }$ is determined by
\begin{equation} \label{eqn: schur function}
 s_{\lambda }(x)
=\sum_{|\mu |=|\lambda|}\frac{\chi_\lambda(C_\mu)}{\mathfrak{z}_\mu}p_\mu(x)
\end{equation}
where $C_\mu$ is the conjugate class of symmetric group corresponding to
partition $\mu $.

\subsection{Partitionable set and infinite series}

The concept of partition can be generalized to the following
\emph{partitionable set}.

\begin{definition}
A countable set $(S, +)$ is called a partitionable set if
\begin{enumerate}
\item[$1)$.] $S$ is totally ordered.

\item[$2)$.] $S$ is an Abelian semi-group with summation $``+"$.

\item[$3)$.] The minimum element $\mathbf{0}$ in $S$ is the zero-element of
the semi-group, i.e., for any $a\in S$,
\[
 \mathbf{0} +a=a=a+\mathbf{0} .
\]
\end{enumerate}
\end{definition}
For simplicity, we may briefly write $S$ instead of $(S, +)$.

\begin{example}
The following sets are examples of partitionable set:
\begin{itemize}
\item[$(1)$.] The set of all nonnegative integers $\mathbb{Z}_{\geq 0}$;

\item[$(2)$.] The set of all partitions $\mathcal{P}$. The order of
$\mathcal{P}$ can be defined as follows:
$\forall \lambda$ , $\mu \in \mathcal{P}$, $\lambda \geq \mu$ iff
$|\lambda| > |\mu|$, or $|\lambda|=|\mu|$ and there exists a $j$
such that $\lambda_i=\mu_i$ for $i\leq j-1$ and $\lambda_j>\mu_j$.
The summation is taken to be $``\cup "$ and the zero-element is $(0)$.

\item[$(3)$.] $\mathcal{P}^{n}$. The order of $\mathcal{P}^{n}$ is defined
similarly as $(2)$:
$\forall \vec{A}, \vec{B}\in \mathcal{P}^{n},\,
 \vec{A}\geq \vec{B}$
iff
$\sum_{i=1}^n |A^i|>\sum_{i=1}^n |B^i|$, or
$\sum_{i=1}^n |A^i| = \sum_{i=1}^n |B^i|$ and there is a $j$ such that
$A^{i} = B^{i}$ for $i\leq j-1$ and $A^j > B^j$.
Define
\[
 \vec{A}\cup \vec{B} = (A^{1}\cup B^{1},...,A^{n}\cup B^{n}).
\]
$((0),(0),...,(0))$ is the zero-element. It's easy to check that
$\mathcal{P}^{n}$
is a partitionable set.
\end{itemize}
\end{example}

Let $S$ be a partitionable set. One can define partition with respect to $S$
in the similar manner as that of $\mathbb{Z}_{\geq 0}$: a finite sequence of
non-increasing non-minimum elements in $S$. We will call it an $S$-partition,
$(\mathbf{0})$ the zero $S$-partition. Denote by $\mathcal{P}(S)$ the set of
all $S$-partitions.

For an $S$-partition $\Lambda$, we can define the automorphism group of
$\Lambda$ in a similar way as that in the definition of traditional partition.
Given $\beta\in S$, denote by $m_\beta(\Lambda)$ the number of times that
$\beta$ occurs in the parts of $\Lambda$, we then have
\[
 \Aut \Lambda = \prod_{\beta\in S} m_\beta(\Lambda)! \,.
\]
Introduce the following quantities associated with $\Lambda$,
\begin{align}\label{def of u-lambda and theta_lambda}
    \mathfrak{u}_\Lambda = \frac{\ell(\Lambda)!}{|\Aut \Lambda|}\,,
\qquad
    \theta_\Lambda = \frac{(-1)^{\ell(\Lambda)-1}} {\ell(\Lambda)}
    \mathfrak{u}_\Lambda \,.
\end{align}

The following Lemma is quite handy when handling the expansion of generating
functions.

\begin{lemma}\label{lemma: taylor expansion}
Let $S$ be a partitionable set. If $f(t)=\sum_{n\geq 0}a_{n}t^{n}$, then
\[
  f\bigg( \sum_{\beta \neq \mathbf{0},\,\beta \in S}
    A_{\beta }\, p_\beta (x)
   \bigg)
= \sum_{\Lambda \in \mathcal{P}(S)} a_{\ell (\Lambda)}
    A_{\Lambda}\,p_{\Lambda }(x)\,\fu_{\Lambda}\,,
\]
where
\begin{align*}
  p_{\Lambda } = \prod_{j=1}^{\ell(\Lambda )}p_{\Lambda _{j}}\,,
    \qquad
  A_{\Lambda} = \prod_{j=1}^{\ell(\Lambda )}A_{\Lambda _{j}}\,.
\end{align*}
\end{lemma}

\begin{proof}
Note that
\[
 \bigg(\sum_{\beta\in S,\,\beta\neq\mathbf{0}}\eta_\beta\bigg)^n
=\sum_{\Lambda\in\mathcal{P}(S),\,\ell(\Lambda)=n}
    \eta_{\Lambda}\fu_{\Lambda}\,.
\]
Direct calculation completes the proof.
\end{proof}

\section{Labastida-Mari\~{n}o-Ooguri-Vafa conjecture}\label{sec: lmov}

\subsection{Quantum trace}

Let $\mathfrak{g}$ be a finite dimensional complex semi-simple Lie algebra of
rank $N$ with Cartan matrix $(C_{ij})$. \emph{Quantized universal enveloping
algebra} $U_q (\mathfrak{g})$ is generated by $\{ H_{i},X_{+i},X_{-i}\}$
together with the following defining relations:
\begin{align*}
& [H_{i},H_{j}]=0\,, \quad [H_{i},X_{\pm j}]=\pm C_{ij}X_{\pm j}\,,\quad
    [X_{+i},X_{-j}]
  = \delta_{ij} \frac{q_i^{-H_i/2}-q_{i}^{H_i/2}}
    {q_i^{-1/2}-q_i^{1/2}}\,,
\\
& \sum_{k=0}^{1-C_{ij}} (-1)^{k}
    \Big\{
        \begin{array}{c}
         1-C_{ij} \\
         k
        \end{array}
    \Big\}_{q_{i}}
    X_{\pm i}^{1-C_{ij}-k}X_{\pm j} X_{\pm i}^{k}=0\,, \quad
    \textrm{for all }i\neq j \,,
\end{align*}
where
\begin{align*}
  \{k\}_{q}
 =\frac{q^{-\frac{k}{2}}-q^{\frac{k}{2}}}
  {q^{-\frac{1}{2}}-q^{\frac{1}{2}}}\,,
    &&
 \{k\}_{q}!=\prod_{i=1}^{k}\{i\}_q\,,
\end{align*}
and
\[
 \Big\{
    \begin{array}{c}
     a \\
     b
    \end{array}
 \Big\}_q
=\frac{\{a\}_{q} \cdot \{a-1\}_{q} \cdots \{a-b+1\}_{q}}{\{b\}_{q}!}\,.
\]

The ribbon category structure associated with $U_q(\mathfrak{g})$ is given
by the following datum:

\begin{enumerate}
\item For any given two $U_q(\mathfrak{g})$-modules $V$ and $W$, there is an
isomorphism
\[
 \chR_{V,\,W}:\, V\otimes W \rightarrow W \otimes V
\]
satisfying
\begin{align*}
   \chR_{U\otimes V,\,W}
&=(\chR_{U,\,W}\otimes\id_V)(\id_U\otimes\chR_{V,\,W})
\\
   \chR_{U,\,V\otimes W}
&=(\id_V\otimes\chR_{U,\,W})(\chR_{U,\,V}\otimes\id_W)
\end{align*}
for $U_q(\mathfrak{g})$-modules $U$, $V$, $W$.

Given $f\in \Hom_{U_q( \mathfrak{g})}(U,\widetilde{U})$,
$g\in \Hom_{U_q(\mathfrak{g})}(V,\widetilde{V})$, one has the following
naturality condition:
\[
 (g\otimes f)\circ\chR_{U,\,V}
=\chR_{\widetilde{U},\,\widetilde{V}}\circ (f\otimes g)  \,.
\]

\item There exists an element
$K_{2\rho}\in U_q(\mathfrak{g})$, called the enhancement  of
$\check{\mathcal{R}}$, such that
\[
 K_{2\rho}(v\otimes w) = K_{2\rho}(v)\otimes K_{2\rho}(w)
\]
for any $v\in V$, $w\in W$.

\item For any $U_q (\mathfrak{g})$-module $V$, the ribbon structure ${%
\Theta }_{V}:V\rightarrow V$ associated to $V$ satisfies
\[
\Theta _{V}^{\pm 1}=\tr_{V}\check{\mathcal{R}}_{V,V}^{\pm 1}.
\]%
The ribbon structure also satisfies the following naturality condition
\[
x\cdot \Theta _{V}=\Theta _{\widetilde{V}}\cdot x.
\]%
for any $x\in \Hom_{U_q (\mathfrak{g})}(V,\widetilde{V})$.
\end{enumerate}

\begin{definition}
Given
$
 z
=\sum_{i}x_i\otimes y_i\in \End_{U_q(\mathfrak{g})}(U\otimes V)
$,
the quantum trace of $z$ is defined as follows
\[
 \tr_{V}(z)
=\sum_{i}\tr(y_i K_{2\rho })x_i\in \End_{U_q(\mathfrak{g})}(U).
\]
\end{definition}

\subsection{Quantum group invariants of links}
\label{subsec: quantum group invariants}

Quantum group invariants of links can be defined over any complex simple Lie
algebra $\mathfrak{g}$. However, in this paper, we only consider the quantum
group invariants of links defined over
$\mathfrak{sl}(N,\mathbb{C})$\footnote{In the following context, we will briefly
write $\fsl_N$.}
due to the current
consideration for large $N$ Chern-Simons/topological string duality.

Roughly speaking, a link is several disconnected $S^1$ embedded in $S^3$.
A theorem of J. Alexander asserts that any oriented link is isotopic to the
closure of some braid. A braid group $\mathcal{B}_n$ is defined by
generators $\sigma_1,\cdots,\sigma_{n-1}$ and defining relation:
\[
 \Big\{
  \begin{array}{ll}
   \sigma_i \sigma_j = \sigma_j \sigma_i \,, &
    \textrm{ if } |i-j|\geq 2\, ;
\\
   \sigma_i \sigma_j \sigma_i = \sigma_j \sigma_i \sigma_j \,, &
    \textrm{ if } |i-j|=1 \, .
  \end{array}
\]

Let $\mathcal{L}$ be a link with $L$ components $\mathcal{K}_{\alpha }$,
$\alpha=1,\ldots,L$, represented by the closure of an element of braid group
$\mathcal{B}_{m}$. We associate each $\mathcal{K}_\alpha$ an irreducible
representation $R_\alpha$ of quantized universal enveloping algebra $U_q
(\fsl_N)$, labeled by its highest weight $\Lambda_{\alpha }$. Denote the
corresponding module by $V_{\Lambda_\alpha}$. The $j$-th strand in the braid
will be associated with the irreducible module $V_{j}=V_{\Lambda _{\alpha }}$,
if this strand belongs to the component $\mathcal{K}_{\alpha }$. The braiding is
defined through the following \emph{universal $R$-matrix} of $U_q(\fsl_N)$
\[
 \mathcal{R}
=q^{\frac{1}{2}\sum_{i,j}C_{ij}^{-1}H_i\otimes H_j}
 \prod_{\textrm{positive root }\beta} \exp_q
 [( 1-q^{-1}) E_\beta\otimes F_\beta]\,.
\]
Here $( C_{ij})$ is the Cartan matrix and
\[
 \exp_{q}(x)
=\sum_{k=0}^{\infty }q^{\frac{1}{4}k(k+1)}\frac{x^{k}}{\{k\}_{q}!}\,.
\]
Define \emph{braiding} by $\check{\mathcal{R}}=P_{12}\mathcal{R}$, where
$P_{12}(v\otimes w)=w\otimes v$.

Now for a given link $\mathcal{L}$ of $L$ components, one chooses a closed braid
representative in braid group $\mathcal{B}_m$ whose closure is $\mathcal{L}$. In
the case of no confusion, we also use $\mathcal{L}$ to denote the chosen braid
representative in $\mathcal{B}_m$. We will associate each crossing by the
braiding defined above. Let $U$, $V$ be two $U_q(\fsl_N)$-modules labeling two
outgoing strands of the crossing, the braiding $\check{R}_{U,V}$ (resp.
$\check{R}_{V,U}^{-1}$) is assigned as in Figure \ref{fig: crossings}.
\begin{figure}[!ht]
\begin{align*}
    \psfrag{A}[][]{$U$}
    \psfrag{B}[][]{$V$}
    \psfrag{P}[][]{$\chR_{U,V}$}
    \includegraphics[width=1.5cm]{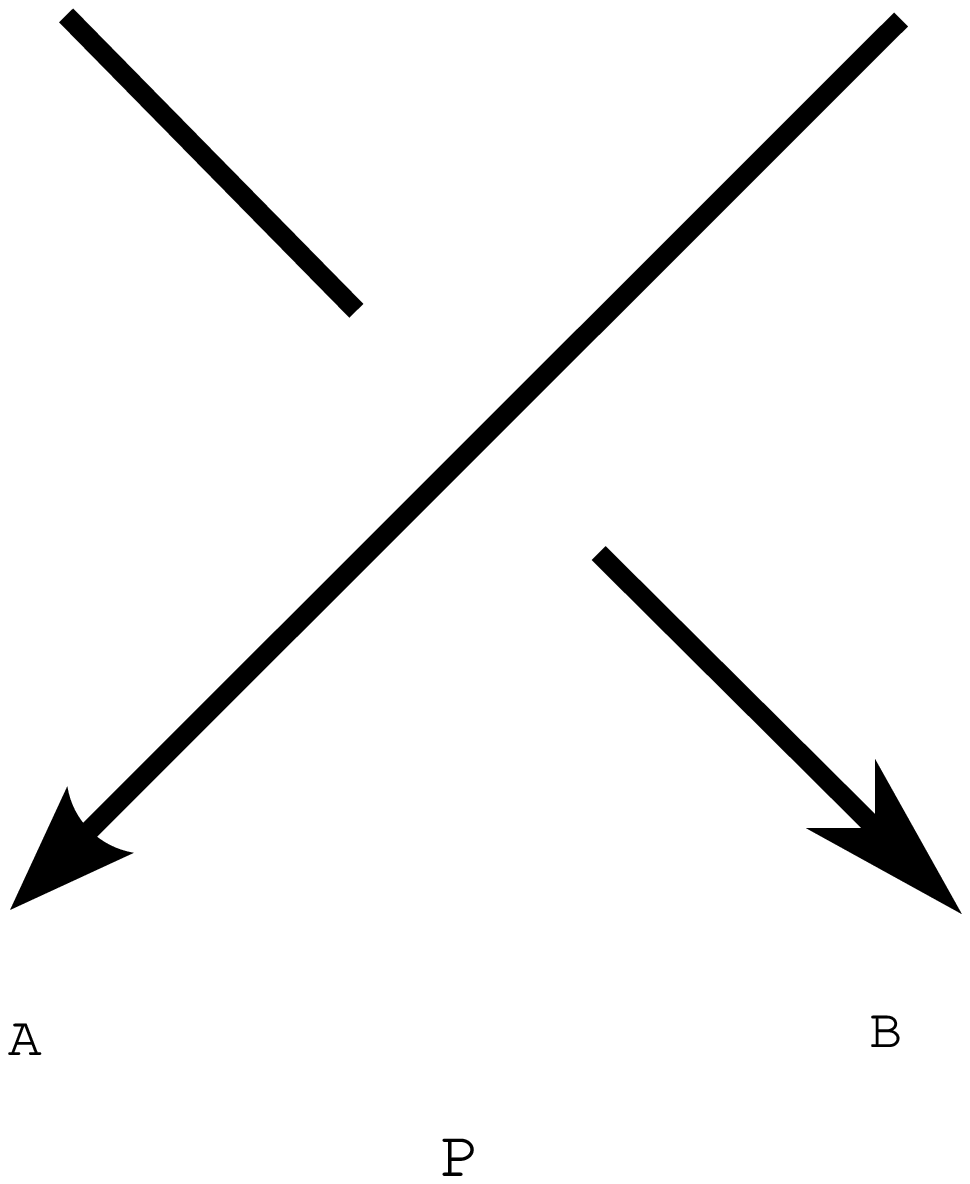}
  &&
    \psfrag{A}[][]{$U$}
    \psfrag{B}[][]{$V$}
    \psfrag{N}[][]{$\chR^{-1}_{V,U}$}
    \includegraphics[width=1.5cm]{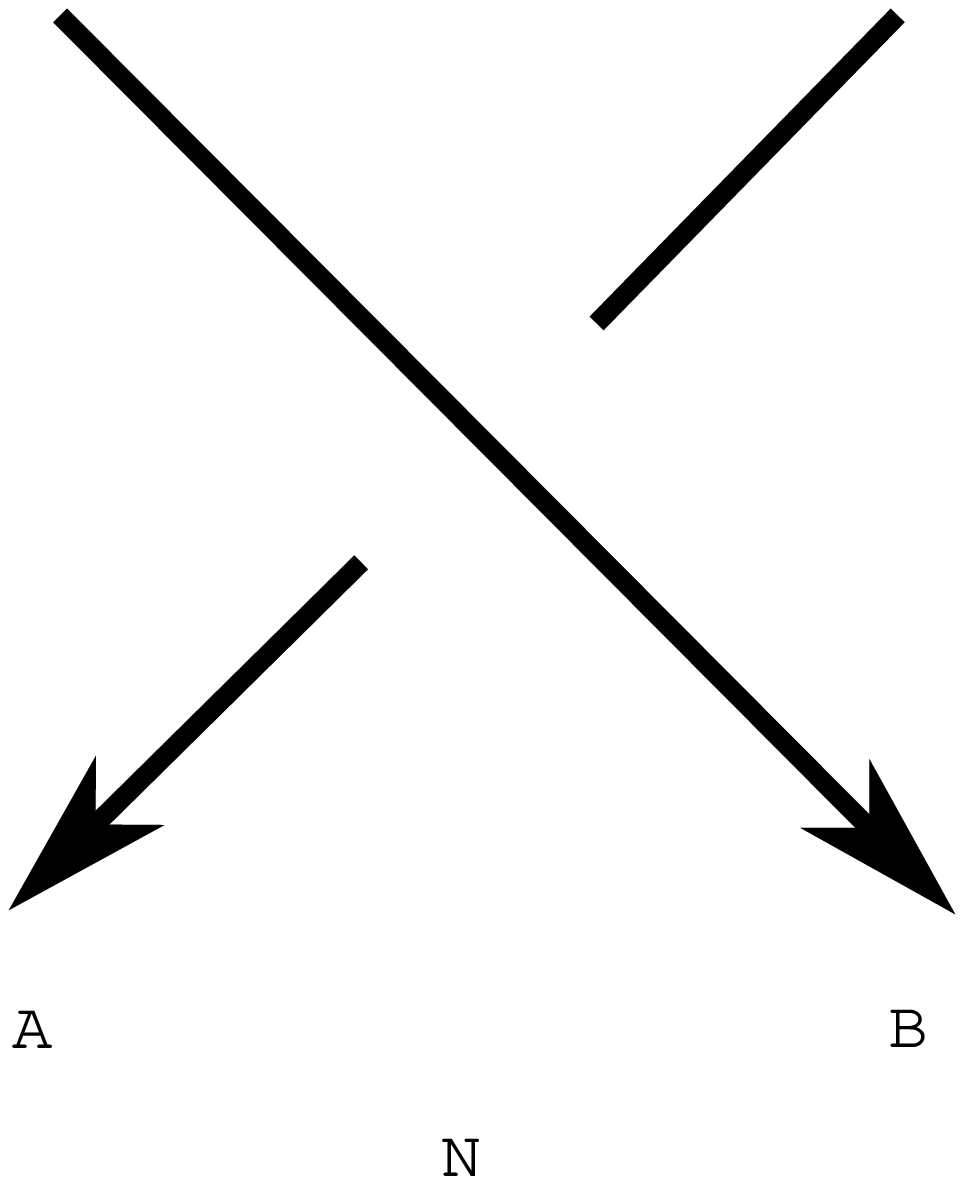}
\end{align*}
\caption{Assign crossing by $\chR$.}\label{fig: crossings}
\end{figure}

The above assignment will give a representation of $\mathcal{B}_m$ on
$U_q(\mathfrak{g})$-module $V_1 \otimes \cdots \otimes V_m$. Namely, for any
generator,
$\sigma_i \in\mathcal{B}_m$\footnote{In the case of $\sigma_i^{-1}$,
use $\chR^{-1}_{V_i, V_{i+1}}$ instead.},
\begin{center}
    \includegraphics[height=1.5cm]{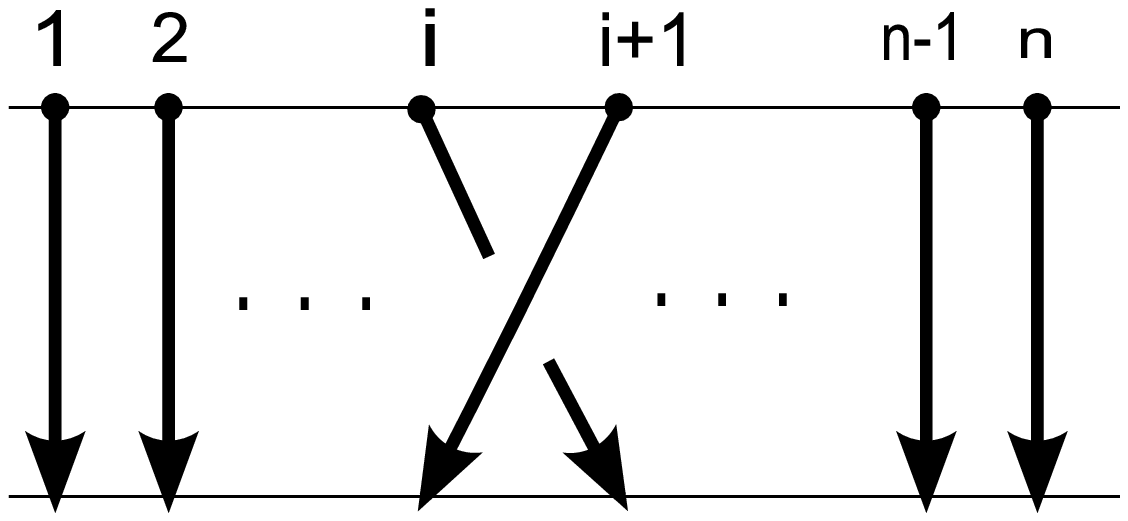}
\end{center}
define
\begin{align*}
 h(\sigma_i)
= \id_{V_1} \otimes \cdots \otimes
 \chR_{V_{i+1}, V_i} \otimes \cdots \otimes id_{V_N}\,.
\end{align*}
Therefore, any link $\mathcal{L}$ will provide an isomorphism
\[
 h(\mathcal{L})
\in \End_{U_q (\fsl_N)}(V_{1}\otimes \cdots \otimes V_{m}) \, .
\]
For example, the link $\mathcal{L}$ in Figure \ref{fig: Hopf link}
gives the following homomorphism
\[
 h(\mathcal{L})
=(\chR_{V,\, U}\otimes \id_U) (\id_V\otimes \chR^{-1}_{U,\, U})
 (\chR_{U,\, V} \otimes \id_U)\, .
\]
\begin{figure}[!ht]
\begin{center}
    \psfrag{U}[][]{$U$}
    \psfrag{V}[][]{$V$}
    \psfrag{L}[][]{$\mathcal{L}$}
    \includegraphics[width=2cm]{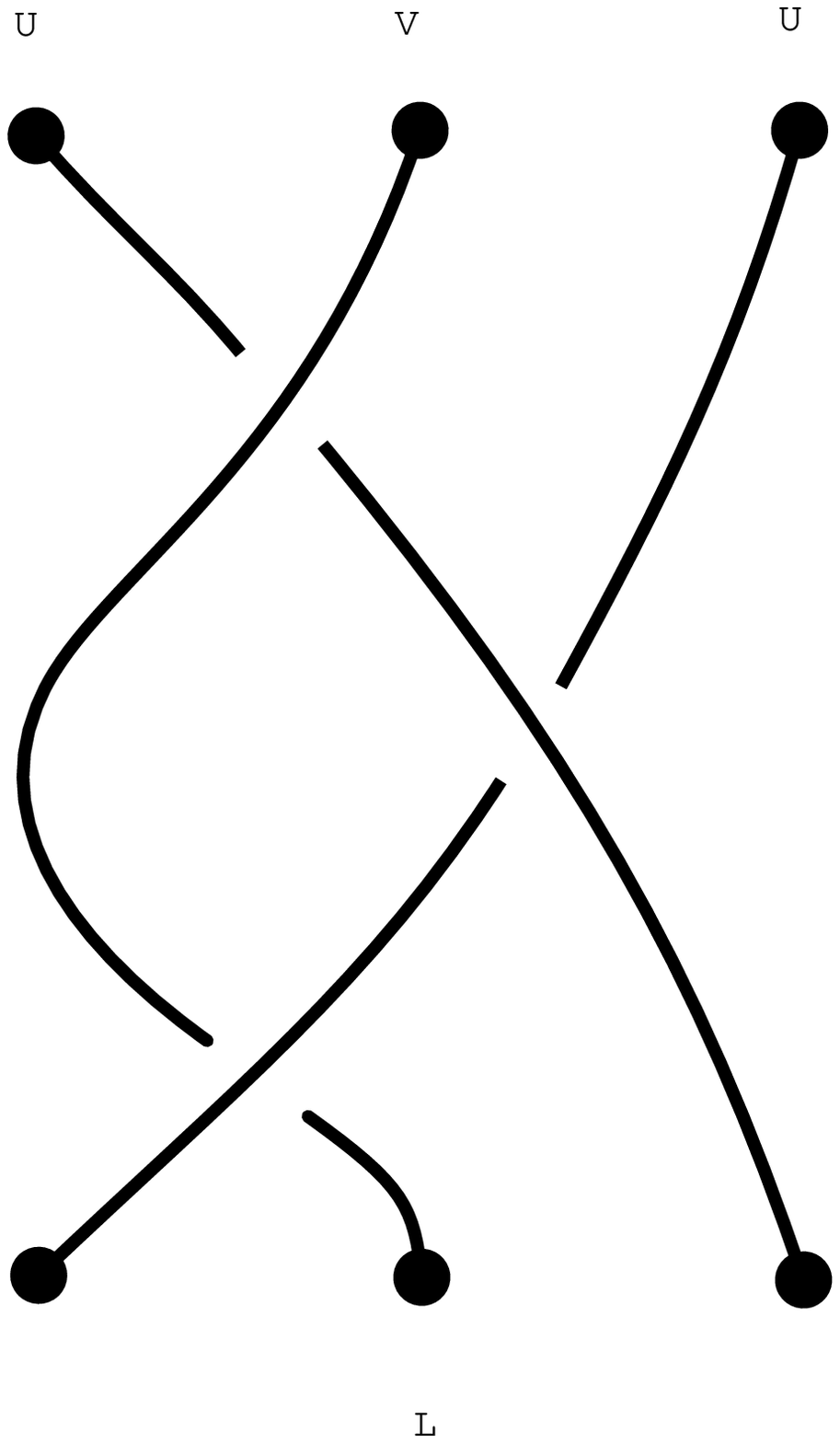}
\end{center}
\caption{A braid representative for Hopf link}\label{fig: Hopf link}
\end{figure}

Let $K_{2\rho}$ be the enhancement of $\check{\mathcal{R}}$ in the sense of
\cite{Rosso-Jones93}, where $\rho $ is the half-sum of all positive roots
of $\fsl_N$. The irreducible representation $R_{\alpha }$ is labeled by the
corresponding partition $A^\alpha$.

\begin{definition}\label{def: quantum group invariant}
Given $L$ labeling partitions $A^1,\ldots,A^L$, the quantum group invariant of
 $\mathcal{L}$ is defined as follows:
\[
    W_{(A^{1},...,A^{L})}(\mathcal{L})
=q^{d(\mathcal{L)}}\tr_{V_{1}
    \otimes\cdots \otimes V_{m}}(h(\mathcal{L})) \, ,
\]
where
\[
 d(\mathcal{L})
=-\frac{1}{2}\sum_{\alpha =1}^{L}\omega (\mathcal{K}_\alpha)
  (\Lambda _{\alpha },\Lambda _{\alpha }+2\rho )
 +\frac{1}{N}\sum_{\alpha<\beta }
  \lk(\mathcal{K}_{\alpha },\mathcal{K}_{\beta })
  |A^{\alpha }|\cdot | A^{\beta }| \,,
\]
and $\lk(\mathcal{K}_\alpha,\mathcal{K}_\beta)$ is the linking number of
components $\mathcal{K}_\alpha$ and $\mathcal{K}_\beta$. A substitution of
$t=q^N$ is used to give a two-variable framing independent link invariant.
\end{definition}

\begin{remark}

In the above formula of $d(\mathcal{L})$, the second term on the right
hand side is meant to cancel not important terms involved with $q^{1/N}$
in the definition.
\end{remark}
It will be helpful to extend the definition to allow some labeling
partition to be the empty partition $(0)$. In this case, the corresponding
invariants will be regarded as the quantum group invariants of the link
obtained by simply removing the components labeled by $(0)$.

A direct computation\footnote{It can also be obtained from the ribbon
structure.} shows that after removing the terms of
$q^{\frac{1}{N}}$, $q^{d(\mathcal{L})}$ can be simplified as
\begin{align}\label{computation of q^d(L)}
 q^{\sum_{\alpha=1}^L \kappa_{A^\alpha} w(\mathcal{K}_\alpha)/2}
\cdot t^{\sum_{\alpha=1}^L |A^\alpha| w(\mathcal{K}_\alpha)/2} \,.
\end{align}

\begin{example}\label{example of quantum group invs}
The following examples are some special cases of quantum group
invariants of links.
\begin{itemize}
\item[$(1)$.] If the link involved in the definition is the unknot $\bigcirc$,
\[
 W_A(\bigcirc;\, q,t)
=\tr_{V_A}(\id_{V_A})
\]
is equal to the quantum dimension of $V_A$ which will be denoted by
$\dim_q V_A$.

\item[$(2)$.] When all the components of $\mathcal{L}$ are associated with
the fundamental representation, i.e., the labeling partition is the unique
partition of $1$, the quantum group invariant of $\mathcal{L}$ is related to the
HOMFLY polynomial of the link, $P_{\mathcal{L}}(q,t)$, in the following way:
\[
 W_{(\partition{1},\cdots,\partition{1})}(\mathcal{L};\, q,t)
=t^{\lk(\mathcal{L})}
 \Bigg(
   \frac{t^{-\frac{1}{2}}-t^{\frac{1}{2}}}{q^{-\frac{1}{2}}-q^{\frac{1}{2}}}
 \Bigg)
 P_{\mathcal{L}}(q,t) \, .
\]

\item[$(3)$.] If $\mathcal{L}$ is a disjoint union of $L$ knots, i.e.,
\[
 \mathcal{L}
=\mathcal{K}_1 \otimes \mathcal{K}_2 \otimes \cdots \otimes\mathcal{K}_L \,,
\]
the quantum group invariants of $\mathcal{L}$ is simply the multiplication of
quantum group invariants of $\mathcal{K}_\alpha$
\[
 W_{(A^1,\ldots,A^L)}(\mathcal{L};\,q,t)
=\prod_{\alpha=1}^L W_{A^\alpha}(\mathcal{K}_\alpha;\,q,t)\,.
\]
\end{itemize}
\end{example}

\subsection{Chern-Simons partition function}

For a given link $\mathcal{L}$ of $L$ components, we will fix the following
notations in this paper. Given
$\lambda\in\mathcal{P}$, $\vec{A}=(A^1,A^2,\ldots,A^L)$,
$\vec{\mu} = (\mu^1,\mu^2,\ldots, \mu^L) \in \mathcal{P}^L$.
Let $x=(x^1,...,x^L)$ where $x^\alpha$ is a series of variables
\[
 x^\alpha=(x^\alpha_1,x^\alpha_2,\cdots)\,.
\]
The following notations will be used throughout the paper:
\begin{align*}
 &
 [n]_q = q^{-\frac{n}{2}}-q^{\frac{n}{2}}\,,
    &&
 [\lambda]_q = \prod_{j=1}^{\ell(\lambda)} [\lambda_j]_q\,,
    &&
 \mathfrak{z}_{\vec{\mu}}=\prod_{\alpha =1}^{L}\mathfrak{z}_{\mu^\alpha}\,,
\\
 &
 |\vec{A}| = ( |A^{1}|,...,|A^{L}|)\,,
    &&
 \parallel\vec{A}\parallel =\sum_{\alpha=1}^{L}|A^\alpha|\,,
    &&
 \ell(\vec{\mu}) = \sum_{\alpha=1}^{L}\ell (\mu ^\alpha) \,,
\\
 &
 \vec{A}^{t} =\big( ( A^{1})^t,\ldots,(A^{L})^{t}\big)\,,
    &&
 \chi _{\vec{A}}\left( \vec{\mu}\right)=\prod_{\alpha =1}^{L}
    \chi _{A^{\alpha }}( C_{\mu ^\alpha})\,,
    &&
 s_{\vec{A}}(x)=\prod_{\alpha=1}^{L}s_{A^\alpha}(x^{\alpha }) \,.
\end{align*}
Denote by
\begin{align} \label{eqn: formula of 1^mu}
 1^{\vec{\mu}}
=(1^{|\mu^1|}, \cdots, 1^{|\mu^L|})\,.
\end{align}

\emph{Chern-Simons partition function} can be defined to be the following
generating function of quantum group invariants of $\mathcal{L}$,
\[
 Z_{\mathrm{CS}}(\mathcal{L})
=1+\sum_{\vec{A}\neq 0}W_{\vec{A}}( \mathcal{L};\,q,t) s_{\vec{A}}(x).
\]
Define \emph{free energy}
\begin{equation}\label{eqn: def of F_mu}
F=\log Z=\sum_{\vec{\mu}\neq 0}F_{\vec{\mu}}p_{\vec{\mu}}\left( x%
\right) .
\end{equation}
Here in the similar usage of notation,
\[
 p_{\vec{\mu}}(x)=\prod_{\alpha=1}^L p_{\mu^\alpha}(x^\alpha)\,.
\]

We rewrite Chern-Simons partition function as
\[
 Z_{\mathrm{CS}}(\mathcal{L})
=1+\sum_{\vec{\mu}\neq 0}Z_{\vec{\mu}}p_{\vec{\mu}}(x)
\]
where
\begin{equation}\label{eqn: Z_mu of W_A}
 Z_{\vec{\mu}}
=\sum_{\vec{A}}\frac{\chi_{\vec{A}}(\vec{\mu})}{\mathfrak{z}_{\vec{\mu}}}
 W_{\vec{A}}.
\end{equation}

By Lemma \ref{lemma: taylor expansion}, we have
\begin{equation}\label{eqn: F_mu of Z_mu}
 F_{\vec{\mu}}
=\sum_{\Lambda \in \mathcal{P}(\mathcal{P}^{L}),\,
    |\Lambda | =\vec{\mu}\in \mathcal{P}^{L}}
 \theta_{\Lambda }Z_{\Lambda } \,.
\end{equation}

\subsection{Main results}\label{subsec: main results}

\subsubsection{Two theorems that answer LMOV conjecture}

Let $P_{\vec{B}}(q,t)$ be the function defined by (\ref{eqn: f_A of P_B}),
which can be determined by the following formula,
\[
 P_{\vec{B}}(q,t)
=\sum_{|\vec{A}|=|\vec{B}|}f_{\vec{A}}(q,t)
 \prod_{\alpha=1}^{L}
 \sum_{\mu} \frac{\chi_{A^\alpha}(C_\mu)\chi_{B^\alpha}(C_\mu)}
 {\mathfrak{z}_{\mu }}
 \prod_{j=1}^{\ell(\mu^\alpha)}
 \frac{1}{q^{-\mu_j/2}-q^{\mu_j/2}}\,.
\]

\begin{Theorem}\label{thm: existence}
There exist topological invariants $N_{\vec{B};\,g,Q}\in \mathbb{Q}$ such that
expansion $(\ref{eqn: P_B})$ holds.
\end{Theorem}

\begin{Theorem}\label{thm: integrality}
Given any $\vec{B}\in\mathcal{P}^L$,
the generating function of $N_{\vec{B};\,g,Q}$,
$P_{\vec{B}}(q,t)$, satisfies
\[
 (q^{-1/2}-q^{1/2})^2 P_{\vec{B}}(q,t) \in
 \mathbb{Z}\big[\big( q^{-1/2}-q^{1/2}\big)^2,t^{\pm 1/2}\big] \,.
\]
\end{Theorem}

It is clear that Theorem \ref{thm: existence} and \ref{thm:
integrality} answered LMOV conjecture. Moreover, Theorem \ref{thm:
integrality} implies, for fixed $\vec{B}$, $N_{\vec{B};\,g,Q}$
vanishes at large genera.

The method in this paper may apply to the general complex simple Lie algebra
$\mathfrak{g}$ instead of only considering $\fsl(N,\mathbb{C})$. We will put
this in our future research. If this is the case, it might require a more
generalized duality picture in physics which will definitely be very interesting
to consider and reveal much deeper relation between Chern-Simons gauge theory
and geometry of moduli space. The extension to some other gauge group has
already been done in \cite{Sinha-Vafa} where non-orientable Riemann surfaces is
involved. For the gauge group $SO(N)$ or $Sp(N)$, a more complete picture had
appeared in the recent work of \cite{BFM1,BFM2} and therein a BPS structure of
the colored Kauffman polynomials was also presented. We would like to see that
the techniques developed in our paper extend to these cases.

The existence of (\ref{eqn: P_B}) has its deep root in the duality
between large $N$ Chern-Simons/topological string duality. As
already mentioned in the introduction, by the definition of quantum
group invariants, $P_{\vec{B}}(q,t)$ might have very high order of
pole at $q=1$, especially when the degree of $\vec{B}$ goes higher
and higher. However, LMOV-conjecture claims that the pole at $q=1$
is at most of order $2$ for any $\vec{B}\in \mathcal{P}^{L}$. Any
term that has power of $q^{-1/2}-q^{1/2}$ lower than $-2$ will be
canceled! Without the motivation of Chern-Simons/topological string
duality, this mysterious cancelation is hardly able to be realized
from knot-theory point of view.

\subsubsection{An application to knot theory}
\label{subsec: application to knot theory}

We now discuss applications to knot theory, following \cite{LMV,LM}.

Consider associating the fundamental representation to each component of the
given link $\mathcal{L}$. As discussed above, the quantum group invariant of
$\mathcal{L}$ will reduce to the classical HOMFLY polynomial
$P_{\mathcal{L}}(q,t)$ of $\mathcal{L}$ except for a universal factor. HOMFLY
has the following expansion
\begin{equation}\label{eqn: formula of HOMFLY}
 P_{\mathcal{L}}(q,t)
=\sum_{g\geq 0}p_{2g+1-L}(t)(q^{-\frac{1}{2}}-q^{\frac{1}{2}})^{2g+1-L}\, .
\end{equation}
The lowest power of $q^{-1/2}-q^{1/2}$ is $1-L$, which was proved in
\cite{LiMi} (or one may directly derive it from Lemma
\ref{lemma: quotient of knot inv}).
After a simple algebra calculation, one will have
\begin{align}\label{eqn: connected homfly expansion}
  F_{(\partition{1},\ldots,\partition{1})}
= \bigg( \frac{t^{-1/2}-t^{1/2}}{q^{-1/2}-q^{1/2}} \bigg)
    \sum_{g\geq 0} \widetilde{p}_{2g+1-L}(t)
    (q^{-\frac{1}{2}}-q^{\frac{1}{2}})^{2g+1-L} \, .
\end{align}
Lemma \ref{lemma: degree} states that
\[
 \widetilde{p}_{1-L}(t)=\widetilde{p}_{3-L}(t)
=\cdots=\widetilde{p}_{L-3}(t)=0 \, ,
\]
which implies that the $p_k(t)$ are completely determined by the HOMFLY
polynomial of its sub-links for $k=1-L, 3-L, \ldots, L-3$.

Now, we only look at $\widetilde{p}_{1-L}(t)=0$. A direct comparison of the
coefficients of $F=\log Z_{\mathrm{CS}}$ immediately leads to the following
theorem proved by Lickorish and Millett \cite{LiMi}.

\begin{theorem}[Lickorish-Millett]
Let $\mathcal{L}$ be a link with $L$ components. Its HOMFLY polynomial
\[
 P_{\mathcal{L}}(q,t)
=\sum_{g\geq 0} p^{\mathcal{L}}_{2g+1-L}(t)
    \Big( q^{-\frac{1}{2}}-q^{\frac{1}{2}} \Big)^{2g+1-L}
\]
satisfies
\[
 p_{1-L}^{\mathcal{L}}(t)
=t^{-\lk}
    \Big( t^{-\frac{1}{2}}-t^{\frac{1}{2}} \Big)^{L-1}
    \prod_{\alpha=1}^L p_0^{\mathcal{K}_\alpha}(t)
\]
where $p_0^{\mathcal{K}_\alpha}(t)$ is HOMFLY polynomial of the $\alpha$-th
component of the link $\mathcal{L}$ with $q=1$.
\end{theorem}

In \cite{LiMi}, Lickorish and Millett obtained the above theorem by
skein analysis. Here as the consequence of higher order cancelation
phenomenon, one sees how easily it can be achieved. Note that we
only utilize the vanishing of $\widetilde{p}_{1-L}$. If one is ready
to carry out the calculation of more vanishing terms, one can
definitely get much more information about algebraic structure of
HOMFLY polynomial. Similarly, a lot of deep relation of quantum
group invariants can be obtained by the cancelation of higher order
poles.

\subsubsection{Geometric interpretation of the new integer invariants}
\label{subsec: geometric interpretation}

The following interpretation is taken in physics literature from string
theoretic point of view \cite{LMV,OV}.

Quantum group invariants of links can be expressed as vacuum expectation value
of Wilson loops which admit a large $N$ expansion in physics. It can also be
interpreted as a string theory expansion. This leads to a geometric description
of the new integer invariants $N_{\vec{B};\,g,Q}$ in terms of open Gromov-Witten
invariants (also see \cite{LM} for more details).

The geometric picture of $f_{\vec{A}}$ is proposed in \cite{LMV}. One can
rewrite the free energy as
\begin{equation}\label{eqn: open GW side free energy}
 F
=\sum_{\vec{\mu}}
  \sqrt{-1}^{\ell(\vec{\mu})} \sum_{g= 0}^\infty
  \lambda^{2g-2+\ell(\vec{\mu})} F_{g,\vec{\mu}}(t)
 p_{\vec{\mu}} \,.
\end{equation}
The quantities $F_{g,\vec{\mu}}(t)$ can be interpreted in terms of the
Gromov-Witten invariants of Riemann surface with boundaries. It was conjectured
in \cite{OV} that for every link $\mathcal{L}$ in $S^3$, one can canonically
associate a lagrangian submanifold $\mathcal{C}_{\mathcal{L}}$ in the
\emph{resolved conifold}
\[
 \mathcal{O}(-1)\oplus\mathcal{O}(-1)\rightarrow \mathbb{P}^1\,.
\]
The first Betti number $b_1(\mathcal{C}_{\mathcal{L}})=L$, the number of
components of $\mathcal{L}$. Let $\gamma_\alpha$, $\alpha=1,\ldots,L$, be
one-cycles representing a basis for $H_1(\mathcal{C}_\mathcal{L},\,\mathbb{Z})$.
Denote by $\mathcal{M}_{g,h,Q}$ the moduli space of Riemann surfaces of genus
$g$ and $h$ holes embedded in the resolved conifold. There are $h_\alpha$ holes
ending on the non-trivial cycles $\gamma_\alpha$ for $\alpha=1,\ldots,L$. The
product of symmetric groups
\[
 \Sigma_{h_1}\times \Sigma_{h_2}\times \cdots \times \Sigma_{h_L}
\]
acts on the Riemann surfaces by exchanging the $h_\alpha$ holes that end on
$\gamma_\alpha$. The integer $N_{\vec{B};\,q,t}$ is then interpreted as
\begin{equation}\label{eqn: geometric interpretation of N_B}
 N_{\vec{B};\,q,t}
=\chi(\mathbf{S}_{\vec{B}} (H^{\ast}(\mathcal{M}_{g,h,Q}) ) )
\end{equation}
where
$
 \mathbf{S}_{\vec{B}}
=\mathbf{S}_{B^1} \otimes \cdots \otimes \mathbf{S}_{B^L},
$
and $\mathbf{S}_{B^\alpha}$ is the Schur functor.

The recent progress in the mathematical definitions of open Gromov-Witten
invariants \cite{Katz-Liu, Li-Song, Liu-Yau} may be used to put the above
definition on a rigorous setting.

\section{Hecke algebra and cabling} \label{sec: Hecke algebra}

\subsection{Centralizer algebra and Hecke algebra representation}

We review some facts about centralizer algebra and Hecke algebra
representation and their relation to the representation of braid group.

Denote by $V$ the fundamental representation of $U_q (\fsl_N)$\footnote{We
will reserve $V$ to denote the fundamental representation of $U_q (\fsl_N)$
from now on.}. Let
\[
 \{K_i^{\pm 1},E_i,F_i:\, 1\leq i\leq N-1\}
\]
be the standard generators of the quantized universal enveloping algebra
$U_q (\fsl_N)$. Under a suitable basis $\left\{ X_{1},...,X_{N}\right\} $ of
$V$, the fundamental representation is given by the following matrices
\begin{align*}
  E_{i} & \longmapsto  E_{i,i+1}
\\
  F_{i} & \longmapsto  E_{i+1,i}
\\
  K_{i} & \longmapsto  q^{-1/2}E_{i,i}+q^{1/2}E_{i+1,i+1}
                    + \sum_{i\neq j}E_{jj}
\end{align*}
where $E_{i,j}$ denotes the $N\times N$ matrix with $1$ at the
$(i,j)$-position and $0$ elsewhere. Direct calculation shows
\begin{equation}\label{eqn: action of K_2rho}
 K_{2\rho}(X_{i})=q^{-\frac{N+1-2i}{2}}X_{i}
\end{equation}
and
\[
 q^{-\frac{1}{2N}}\check{\mathcal{R}}(X_{i}\otimes X_{j})
=\left\{
 \begin{array}{ll}
    q^{-1/2}X_{i}\otimes X_{j}, & i=j\, , \\
    X_{j}\otimes X_{i}, & i<j\, , \\
    X_{j}\otimes X_{i}+(q^{-1/2}-q^{1/2})X_{i}\otimes X_{j},  & i>j\, .
 \end{array}
 \right.
\]

The \emph{centralizer algebra} of $V^{\otimes n}$ is defines as
follows
\[
 \mathcal{C}_{n}
=\End_{U_q (\fsl_N)}(V^{\otimes n})
=\left\{ x\in\End( V^{\otimes n}) :
 \, xy=yx,\forall y\in U_q (\fsl_N)\right\} \,.
\]

\emph{Hecke algebra} $\mathcal{H}_n(q)$ of type $A_{n-1}$ is the complex algebra
with $n-1$ generators $g_{1},...,g_{n-1}$, together with the following defining
relations
\begin{align*}
  &  g_{i}g_{j}=g_{j}g_i\,,
    && | i-j| \geq 2
\\
  &  g_{i}g_{i+1}g_{i}=g_{i+1}g_{i}g_{i+1}\,,
    && i=1,2,...,n-2,
\\
  &  (g_{i}-q^{-1/2})(g_{i}+q^{1/2})=0\,,
    && i=1,2,...,n-1.
\end{align*}
\begin{remark}\label{remark: Hecke algebra definition}
Here we use $q^{-1/2}$ instead of $q$ to adapt to our notation in the definition
of quantum group invariants of links. Note that when $q=1$, the Hecke algebra
$\mathcal{H}_n(q)$ is just the group algebra $\mathbb{C}\Sigma_n$ of symmetric
group $\Sigma_n$. When $N$ is large enough, $\mathcal{C}_n$ is isomorphic to
the Hecke algebra $\mathcal{H}_n(q)$.
\end{remark}

A very important feature of the homomorphism
\begin{align*}
 h:\,\mathbb{C}\mathcal{B}_{n}\longrightarrow \mathcal{C}_{n}
\end{align*}
is that $h$ factors through
$\mathcal{H}_{n}(q)$ via
\begin{align}\label{h factors through Hecke algebra}
 q^{-\frac{1}{2N}}\sigma_i \mapsto g_i \mapsto
 q^{-\frac{1}{2N}}h(\sigma_i) \, .
\end{align}

It is well-known that the irreducible modules $S^{\lambda }$
(\emph{Specht module})
of $\mathcal{H}_{n}(q)$ are in one-to-one correspondence to the partitions
of $n$.

Any permutation $\pi$ in symmetric group $\Sigma_n$
can express as a product of transpositions
\[
  \pi =s_{i_{1}}s_{i_{2}}\cdots s_{i_l}\,.
\]
If $l$ is minimal
in possible, we say $\pi $ has length $\ell(\pi)=l$ and
\[
 g_\pi =g_{i_{1}}g_{i_{2}}\cdots g_{i_{l}}\,.
\]
It is not difficult to see that $g_\pi$ is well-defined.
All of such $\{g_\pi\}$ form a basis of $\mathcal{H}_n(q)$.

Minimal projection $\fS$ is an element in $\mathcal{C}_{n}$ such that $\fS V^{\otimes
n}$ is some irreducible representation $\fS_{\lambda }$. We denote it by
$p_{\lambda }$. The minimal projections of Hecke algebras are well studied (for
example \cite{Gy}), which is a $\mathbb{Z}(q^{\pm 1})$-linear combination
$\{q^{\frac{1}{2}} g_i\}$.

\subsection{Quantum dimension}
\label{subsec: quantum dimension}

\subsubsection{Explicit formula}

An explicit formula for quantum dimension of any irreducible representation of
$U_q(\fsl_N)$ can be computed via decomposing $V^{\otimes n}$ into permutation
modules.

A \emph{composition} of $n$ is a sequence of non-negative integer
\[
 \mathfrak{b}=(\mathfrak{b}_1,\mathfrak{b}_2,\ldots)
\]
such that
\[
 \sum_{i\geq 1}\mathfrak{b}_i=n \,.
\]
We will write it as $\mathfrak{b}\vDash n$. The largest $j$ such that
$\mathfrak{b}_j\neq 0$ is called the \emph{end} of $\mathfrak{b}$ and denoted by
$\ell(\mathfrak{b})$.

Let $\mathfrak{b}$ be a composition such that $\ell(\mathfrak{b})\leq N$. Define
$M^{\mathfrak{b}}$ to be the subspace of $V^{\otimes n}$ spanned by the vectors
\begin{equation}\label{basis of V}
 X_{j_1}\otimes\cdots \otimes X_{j_n}
\end{equation}
such that $X_i$ occurs precisely $\mathfrak{b}_i$ times. It is clear that
$M^{\mathfrak{b}}$ is an $\mathcal{H}_n(q)$-module and is called
\emph{permutation module}. Moreover, by explicit matrix formula of
$\{E_i,F_i,K_i\}$ acting on $V$ under the basis $\{X_i\}$, we have
\[
 M^{\mathfrak{b}}
=\{ X\in V^{\otimes n}:\,
    K_i(X)=q^{-\frac{\mathfrak{b}_i-\mathfrak{b}_{i+1}}{2}}X \} \,.
\]
The following decomposition is very useful
\[
 V^{\otimes n}
=\bigoplus_{\mathfrak{b}\vDash n,\, \ell(\mathfrak{b})\leq N}
    M^{\mathfrak{b}} \,.
\]

Let $A$ be a partition and $V_A$ the irreducible representation labeled by
$A$. The \emph{Kostka number} $K_{A\mathfrak{b}}$ is defined to be the weight of
$V_A$ in $M^{\mathfrak{b}}$, i.e.,
\[
 K_{A\mathfrak{b}}
=\dim (V_A \cap M^{\mathfrak{b}}) \,.
\]
Schur function has the following formulation through Kostka numbers
\[
 s_A(x_1,\ldots, x_N)
=\sum_{\mathfrak{b}\vDash |A|,\, \ell(\mathfrak{b})\leq N}
 K_{A\mathfrak{b}} \prod_{j=1}^N x_j^{\mathfrak{b}_j} \,.
\]

By (\ref{eqn: action of K_2rho}), $K_{2\rho}$ is acting on $M^{\mathfrak{b}}$ as
a scalar $\prod_{j=1}^N q^{-\frac{1}{2}(N+1-2j)\mathfrak{b}_j}$. Thus
\begin{align}
  \dim_q V_A
&= \tr_{V_A}{\id_{V_A}} \nonumber
\\
&= \sum_{\mathfrak{b}\vDash |A|} \dim (V_A \cap M^{\mathfrak{b}})
    \prod_{j=1}^N q^{-\frac{1}{2}(N+1-2j)\mathfrak{b}_j}
    \nonumber
\\
&= s_A\Big(q^{\frac{N-1}{2}},\ldots, q^{-\frac{N-1}{2}}\Big)
    \label{eqn: quantum dimension as Schur function}
\end{align}
By \eqref{eqn: schur function},
\begin{align}
\dim_q V_A
&= \sum_{|\mu|=|A|} \frac{\chi_A(C_\mu)}{\mathfrak{z}_\mu}
    p_\mu\Big(q^{\frac{N-1}{2}},\ldots, q^{-\frac{N-1}{2}}\Big) \nonumber
\\
&= \sum_{|\mu|=|A|} \frac{\chi_A(C_\mu)}{\mathfrak{z}_\mu}
    \prod_{j=1}^{\ell(\mu)}
    \frac{t^{-\mu_j/2}-t^{\mu_j/2}}{q^{-\mu_j/2}-q^{\mu_j/2}}
    \label{eqn: formula of quantum dimension} \,.
\end{align}
Here in the last step, we use the substitution $t=q^N$.

\subsubsection{An expansion of the Mari\~no-Vafa formula}
\label{subsec: mv-formula}

Here we give a quick review about Mari\~no-Vafa formula \cite{MV,
LLZ1} for the convenience of Knot theorist. For details, please
refer \cite{LLZ1}.

Let $\Mbar_{g,n}$ denote the Deligne-Mumford moduli stack of stable
curves of genus $g$ with $n$ marked points. Let
$\pi:\Mbar_{g,n+1}\to \Mbar_{g,n}$ be the universal curve, and let
$\omega_\pi$ be the relative dualizing sheaf. The Hodge bundle
$\mathbb{E}=\pi_*\omega_\pi$ is a rank $g$ vector bundle over
$\Mbar_{g,n}$. Let $s_i:\Mbar_{g,n}\to\Mbar_{g,n+1}$ denote the
section of $\pi$ which corresponds to the $i$-th marked point, and
let $\mathbb{L}_i=s_i^*\omega_\pi$. A Hodge integral is an integral
of the form
$$\int_{\Mbar_{g, n}} \psi_1^{j_1}
\cdots \psi_n^{j_n}\lambda_1^{k_1} \cdots \lambda_g^{k_g}$$ where
$\psi_i=c_1(\mathbb{L}_i)$ is the first Chern class of
$\mathbb{L}_i$, and $\lambda_j=c_j(\mathbb{E})$ is the $j$-th Chern
class of the Hodge bundle. Let
\[
 \Lambda^{\vee}_g(u)=u^g-\lambda_1 u +\cdots+(-1)^g\lambda_g
\]
be the Chern polynomial of $\mathbb{E}^\vee$, the dual of the Hodge
bundle.

Define
\begin{align*}
 \mathcal{C}_{g, \mu}(\tau)
& =  - \frac{\sqrt{-1}^{\ell(\mu)}}{|\Aut(\mu)|}
    [\tau(\tau+1)]^{\ell(\mu)-1}
    \prod_{i=1}^{\ell(\mu)}\frac{ \prod_{a=1}^{\mu_i-1} (\mu_i \tau+a)}{(\mu_i-1)!}
\\
&   \qquad
    \cdot \int_{\Mbar_{g, l(\mu)}}
    \frac{\Lambda^{\vee}_g(1)\Lambda^{\vee}_g(-\tau-1)\Lambda_g^{\vee}(\tau)}
    {\prod_{i=1}^{\ell(\mu)}(1- \mu_i \psi_i)}\,.
\end{align*}
Note that
\begin{align}
 \mathcal{C}_{0,\mu}(\tau)
& = - \frac{\sqrt{-1}^{\ell(\mu)}}{|\Aut(\mu)|}
    [\tau(\tau+1)]^{\ell(\mu)-1}
    \prod_{i=1}^{\ell(\mu)}\frac{ \prod_{a=1}^{\mu_i-1} (\mu_i \tau+a)}{(\mu_i-1)!}
    \nonumber
\\
&   \qquad
    \cdot \int_{\Mbar_{0, l(\mu)}}
    \frac{1}{\prod_{i=1}^{\ell(\mu)}(1- \mu_i \psi_i)}
    \label{eqn: C_0mu}
\,.
\end{align}
The coefficient of the leading term in $\tau$ is:
\begin{align}
 -\frac{\sqrt{-1}^{\ell(\mu)}}{|\Aut \mu|}
    \prod_j\frac{ \mu_j^{\mu_j}}{\mu_j!}\cdot |\mu|^{\ell(\mu)-3}\,.
    \label{formula: leading term in tau in C_0,mu}
\end{align}

The Mari\~no-Vafa formula gives the following identity:
\begin{align}\label{eqn: MV formula}
 \sum_{\mu}p_\mu(x) \sum_{g\geq 0} u^{2g-2+\ell(\mu)}
    \mathcal{C}_{g,\mu}(\tau)
=\log \Big( 1+ \sum_{A} s_A(q^\rho) s_A(x) \Big)
\end{align}
where $q^\rho=(q^{-1/2}, q^{-3/2},\cdots,q^{-n+1/2},\cdots)$.

Let
\[
 \log \Big( \sum_A s_A(q^\rho) s_A(y) q^{\frac{\kappa_A \tau}{2}} \Big) = \sum_{\mu}
G_\mu p_\mu.
\]
Then
\begin{align}
 G_\mu
&=\sum_{|\Lambda|=\mu} \Theta_\Lambda
\prod_{\alpha=1}^{\ell(\Lambda)}
    \sum_{A^\alpha} \frac{\chi_{A^\alpha}(\Lambda^\alpha)}{z_{\Lambda^\alpha}}
    s_{A^{\alpha}}(q^\rho) q^{\frac{\kappa_{A^\alpha}\tau}{2}}
        \label{eqn: convolution for G_mu in MV-formula}
\\
&=G_\mu(0) + \sum_{p\geq 1} \frac{(\frac{u\tau}{2})^p}{p!}
    \sum_{|\Lambda|=\mu} \Theta_\Lambda \sum_{\Omega\neq (1^\Lambda),
    \mathcal{A}} \frac{\chi_{\mathcal{A}}(\Lambda)}{z_\Lambda}
    \frac{\chi_{\mathcal{A}}(\Omega)}{z_\Omega} p_\Omega(q^\rho)
    \kappa_{\mathcal{A}}^p
        \nonumber
\\
& \qquad
    +\frac{1}{(q^{\frac{1}{2}}-q^{-\frac{1}{2}})^{|\mu|}} \sum_{p\geq 1}
    \frac{(\frac{u\tau}{2})^p}{p!} \sum_{|\Lambda|=\mu} \Theta_\Lambda
    \sum_{\mathcal{A}}
    \frac{\chi_{\mathcal{A}}(\Lambda)}{z_{\Lambda}}
    \frac{\chi_{\mathcal{A}}(1^\Lambda)}{z_{1^\Lambda}} \kappa_A^p
        \label{eqn: leading term in G_mu in MV-formula}
\end{align}
The third summand of the above formula gives the non-vanishing
leading term in $\tau$ which is equal to \eqref{formula: leading
term in tau in C_0,mu}. Therefore, we have:
\begin{align}\label{ineq: deg_u F_mu = L_mu-2 in MV-formula}
 \sum_{|\Lambda|=\mu} \Theta_\Lambda
    \sum_{\mathcal{A}}
    \frac{\chi_{\mathcal{A}}(\Lambda)}{z_{\Lambda}}
    \frac{\chi_{\mathcal{A}}(1^\Lambda)}{z_{1^\Lambda}} \kappa_A^p
\neq 0
\end{align}
for $\forall p\geq |\mu|+\ell(\mu)-2$.

\subsection{Cabling Technique}

Given irreducible representations $V_{A^{1}},...,V_{A^{L}}$ to each component of
link $\mathcal{L}$. Let $| A^{\alpha }|=d_{\alpha }$,
$\vec{d}=(d_{1},...,d_{L})$. The cabling braid of $\mathcal{L}$,
$\mathcal{L}_{\vec{d}}$, is obtained by substituting $d_\alpha$ parallel strands
for each strand of $\mathcal{K}_{\alpha }$, $\alpha=1,\ldots,L$.

Using cabling of the $\mathcal{L}$ gives a way to take trace in the vector space
of tensor product of fundamental representation. To get the original trace, one
has to take certain projection, which is the following lemma in \cite{LZ}.

\begin{lemma}[\cite{LZ}, Lemma 3.3]\label{lemma: lin-zheng}
Let $V_{i}=\fS_{i}V^{\otimes d_{i}}$ for some minimal projections,
$\fS_{i}=\fS_{j} $ if the $i$-th and $j$-th strands belong to the same knot. Then
\[
 \tr_{V_{1}\otimes \cdots V_{m}}\left( h(\mathcal{L}) \right)
=\tr_{V^{\otimes n}}(h(\mathcal{L}_{(d_{1},...,d_{L})})
 \circ \fS_{1}\otimes \cdots \otimes \fS_{n})\, .
\]
where $m$ is the number of strands belonging to $\mathcal{L}$,
$n=\sum_{\alpha=1}^{L}d_{\alpha }r_{\alpha }$, $r_{\alpha }$ is the number of
strands belong to $\mathcal{K}_\alpha$, the $\alpha$-th component of
$\mathcal{L}$.
\end{lemma}

\section{Proof of Theorem \ref{thm: existence}} \label{sec: existence}

\subsection{Pole structure of quantum group invariants}

By an observation from the action of $\chR$ on $V\otimes V$, we define
\[
 \widetilde{X}_{(i_{1},...,i_{n})}
=q^{\frac{\#\{(j,k)|j<k,i_j>i_k\}}{2} }
    X_{i_{1}}\otimes \cdots \otimes X_{i_{n}}.
\]
$\{\widetilde{X}_{(i_{1},...,i_{n})}\}$ form a basis of
$V^{\otimes n}$. By \eqref{h factors through Hecke algebra},
 $\forall g_j \in \mathcal{H}_n(q)$, we have
\begin{align}
  q^{\frac{1}{2}} g_j \widetilde{X}_{(...,i_{j},i_{j+1},...)}
     =
 \left\{
    \begin{array}{ll}
  \widetilde{X}_{(...,i_{j+1},i_{j},...),}
    & i_{j}\leq i_{j+1}\, ,
\\
  q\widetilde{X}_{(...,i_{j+1},i_{j},...)}+(1-q)
    \widetilde{X}_{(...,i_{j,}i_{j+`},...),\ }
    & i_{j}\geq i_{j+1}.
    \end{array}
\right. \label{action of g_j on special basis}
\end{align}

\begin{lemma}\label{lemma: quotient of knot inv}
Let $\bigcirc$ be the unknot. Given any
$\vec{A} =(A^1,\ldots,A^L) \in \mathcal{P}^L$,
\begin{equation}\label{eqn: normalization of W_A in degree u}
 \lim_{q\rightarrow 1}
 \frac{W_{\vec{A}}(\mathcal{L};\,q,t)}{W_{\vec{A}}(\bigcirc ^{\otimes L};\,q,t)}
=\prod_{\alpha=1}^L
 \xi_{\mathcal{K}_\alpha}(t)^{d_\alpha} \,,
\end{equation}
where $|A^\alpha|=d_\alpha$, $\mathcal{K}_\alpha$ is the
$\alpha$-th component of $\mathcal{L}$,
and $\xi_{\mathcal{K}_\alpha}(t)$, $\alpha=1,\ldots,L$, are
independent of $\vec{A}$.
\end{lemma}

\begin{proof}
Choose $\beta\in\mathcal{B}_m$ such that $\mathcal{L}$ is the closure of
$\beta$, the total number of crossings of $\mathcal{L}_{\vec{d}}$ is even and
the last $L$ strands belongs to distinct $L$ components of $\mathcal{L}$. Let
$r_\alpha$ be the number of the strands which belong to $\mathcal{K}_{\alpha }$.
$n=\sum_\alpha d_\alpha r_\alpha$ is equal to the number of components in the
cabling link $\mathcal{L}_{\vec{d}}$.

Let
\begin{equation} \label{eqn: def of Y}
 \mathcal{Y}
=\tr_{V^{\otimes \sum_{\alpha=1}^L d_\alpha (r_\alpha-1)}}
    \big( \mathcal{L}_{\vec{d}}\big) \,.
\end{equation}
$\mathcal{Y}$
is both a central element of
$\End_{U_q (\fsl_N)}\big( V^{\otimes(d_{1}+...+d_{L})}\big)$
and a
$\mathbb{Z}[q^{\pm 1}]$-matrix under the
basis $\{ \widetilde{X}_{(i_1,\ldots,i_n)}\} $.

On the other hand,
\[
 p_{\vec{A}}=p_{A^{1}}\otimes ...\otimes p_{A^{L}}
\]
is a
$\mathbb{Z}\left( q^{\pm 1}\right)$-matrix under the basis
$\{ \widetilde{X}_{(i_1,\ldots, i_n)} \} $.
By Schur lemma, we have
\[
 p_{\vec{A}}\circ \mathcal{Y}=\epsilon \cdot p_{\vec{A}} \,,
\]
where $\epsilon \in \mathbb{Z}(q^{\pm 1})$ is an
eigenvalue of $\mathcal{Y}$.
Since $\mathbb{Z}[ q^{\pm 1}] $ is a UFD for
transcendental $q$, $\epsilon$ must stay in
$\mathbb{Z}[ q^{\pm 1}]$. By Lemma \ref{lemma: lin-zheng},
\begin{align}
   \tr_{V^{\otimes n}}(p_{A^1}^{\otimes r_1}\otimes\cdots
   \otimes p_{A^L}^{\otimes r_L}\circ \mathcal{L}_{\vec{d}})
&=\tr_{V^{\otimes (d_{1}+\cdots d_{L})}}
    (p_{\vec{A}}\circ\mathcal{Y}) \nonumber
\\
&=\epsilon \cdot \tr_{V^{^{\otimes (d_{1}+\cdots d_{L})}}}
    ( p_{\vec{A}}) \nonumber
\\
&=\epsilon \cdot W_{\vec{A}}( \bigcirc ^{\otimes L}) \,.
    \label{eqn: eigenvalue of Y with quantum dimension}
\end{align}
The definition of quantum group invariants of $\mathcal{L}$ gives
\begin{align}\label{eqn: W_A involved with eigenvalue of Y}
 W_{\vec{A}}(\mathcal{L};\,q,t)
=q^{\sum_\alpha \kappa_{A^\alpha} w(\mathcal{K}_\alpha)/2}
 t^{\sum_\alpha d_\alpha w(\mathcal{K}_\alpha)/2} \cdot
 \epsilon\cdot W_{\vec{A}}(\bigcirc^{\otimes L};\,q,t)
\end{align}

When $q\rightarrow 1$, $\mathcal{L}_{\vec{d}}$ reduces to an element
in symmetric group $\Sigma_{n}$ of $\parallel\vec{d}\parallel$
cycles. Moreover, when $q\rightarrow 1$, the calculation is actually
taken in individual knot component while the linking of different
components have no effect. By Example \ref{example of quantum group
invs} $(1)$ and $(3)$, we have
\[
 W_{\vec{A}}(\bigcirc^{\otimes L};\,q,t)
=\prod_{\alpha=1}^L W_{A^\alpha}(\bigcirc;\,q,t) \,.
\]
This implies
\begin{equation}\label{eqn: eigenvalue of Y as product formula}
 \lim_{q\rightarrow 1}
 \frac{W_{\vec{A}}(\mathcal{L};\,q,t)}
    {W_{\vec{A}}(\bigcirc ^{\otimes L};\,q,t)}
=\prod_{\alpha=1}^L
 \lim_{q\rightarrow 1}
 \frac{W_{A^\alpha}(\mathcal{K}_\alpha;\,q,t)}
    {W_{A^\alpha}(\bigcirc;\,q,t)} \,.
\end{equation}

Let's consider the case when $\mathcal{K}$ is a knot. $A$ is the partition of
$d$ associated with $\mathcal{K}$ and $\mathcal{K}_d$ is the cabling of
$\mathcal{K}$. Each component of $\mathcal{K}_d$ is a copy of $\mathcal{K}$.
$q\rightarrow 1$, $\mathcal{K}_d$ reduces to an element in $\Sigma_{dr}$. To
calculate the $\mathcal{Y}$, it is then equivalent to discussing $d$ disjoint
union of $\mathcal{K}$.
Say $\mathcal{K}$ has $r$ strands. Consider
\[
 \mathcal{Y}_0=\tr_{V^{\otimes(r-1)}} \mathcal{K} \,.
\]
The eigenvalue of $\mathcal{Y}_0$ is then
$\frac{P_{\mathcal{K}}(1,t)}{t^{w(\mathcal{K})}}$, where
$P_{\mathcal{K}}(q,t)$ is the HOMFLY polynomial for $\mathcal{K}$.
Denote by $\xi_{\mathcal{K}}(t)=P_{\mathcal{K}}(1,t)$, we have
\begin{align}\label{eqn: definition of xi_A}
 \lim_{q\to 1} \frac{W_A(\mathcal{K};\,q,t)}{W_A(\bigcirc;\,q,t)}
=\xi_{\mathcal{K}}(t)^{|A|}\,.
\end{align}

Combined with
\eqref{eqn: eigenvalue of Y as product formula}, the proof is completed.
\end{proof}

\subsection{Symmetry of quantum group invariants}

Define
\[
 \phi _{\vec{\mu}}\left( q\right)
=\prod_{\alpha=1}^L\prod_{j=1}^{\ell(\mu^\alpha)}
    (q^{-\mu_j^\alpha/2}-q^{\mu_j^\alpha/2}) \,.
\]
Comparing (\ref{eqn: def of F_mu})
and (\ref{eqn: formula of f_A}), we have
\begin{align*}
   F_{\vec{\mu}}
&=\sum_{d|\vec{\mu}}\frac{1}{d}\sum_{\vec{A}}\frac{\chi_{\vec{A}}
    \left( \vec{\mu}/d\right) }{\mathfrak{z}_{\vec{\mu}/d}}
    f_{\vec{A}}\left( q^{d},t^{d}\right)
\\
&=\sum_{d|\vec{\mu}}\frac{1}{d\cdot\mathfrak{z}_{\vec{\mu}/d}}
    \sum_{\vec{A}}\chi_{\vec{A}}(\vec{\mu}/d)
    \sum_{\vec{B}}P_{\vec{B}}(q^{d},t^{d})
    \prod_{\alpha =1}^L M_{A^\alpha B^\alpha}(q^d)
    \qquad
\\
&=\sum_{d|\vec{\mu}}\frac{1}{d}\cdot
    \frac{\phi_{\vec{\mu}/d}(q^{d}) }{\mathfrak{z}_{\vec{\mu}/d}}
    \sum_{\vec{B}}\chi_{\vec{B}}(\vec{\mu}/d)P_{\vec{B}}(q^{d},t^{d})
\\
&={\phi_{\vec{\mu}}(q)}\sum_{d|\vec{\mu}} \frac{1}{d\cdot
    z_{\vec{\mu}/d} }\sum_{\vec{B}}\chi_{\vec{B}}(\vec{\mu}/d) P_{\vec{B}}(q^d,t^d)
    \,,
\end{align*}
and
\begin{align}
 \frac{F_{\vec{\mu}}}{\phi_{\vec{\mu}}(q)}
=\sum_{d|\vec{\mu}} \frac{1}{d\cdot  z_{\vec{\mu}/d} }
    \sum_{\vec{B}}\chi_{\vec{B}}(\vec{\mu}/d) P_{\vec{B}}(q^d,t^d)
    \label{eqn: F_mu of P_B}
\,.
\end{align}
Apply M\"obius inversion formula,
\begin{equation}\label{eqn: P_B of F_mu}
 P_{\vec{B}}(q,t)
=\sum_{\vec{\mu}}
  \frac{\chi_{\vec{B}}(\vec{\mu})}{\phi_{\vec{\mu}}(q)}
 \sum_{d|\vec{\mu}}
  \frac{\mu (d)}{d}F_{\vec{\mu}/d}(q^{d},t^{d})
\end{equation}
where $\mu (d)$ is the M\"obius function defined as follows
\[
  \mu(d)
= \left\{
    \begin{array}{ll}
     (-1)^{r},  &
        \textrm{if $d$ is a product of $r$ distinct prime numbers;}
     \\
     0,  &
        \textrm{otherwise.}%
    \end{array}
 \right.
\]
To prove the existence of formula (\ref{eqn: P_B}), we need to prove:

\begin{itemize}
\item Symmetry of $P_{\vec{B}}(q,t)=P_{\vec{B}}(q^{-1},t)$.

\item The lowest degree $\left( q^{-1/2}-q^{1/2}\right)$ in
$P_{\vec{B}} $ is no less than $-2$.
\end{itemize}

Combine \eqref{eqn: F_mu of P_B}, \eqref{eqn: F_mu of Z_mu} and
\eqref{eqn: Z_mu of W_A}, it's not difficult to find that the first property on
the symmetry of $P_{\vec{B}}$ follows from the following lemma.

\begin{lemma}\label{lemma: symmetry}
$W_{\vec{A}^{t}}(q,t)=(-1)^{\parallel\vec{A}\parallel}W_{\vec{A}}(q^{-1},t)$.
\end{lemma}

\begin{proof}
The following irreducible decomposition of $U_q(\fsl_N)$ modules is well-known:
\[
 V^n = \bigoplus_{B\vdash n} d_B V_B\,,
\]
where $ d_B = \chi_B (C_{(1^n)})$.

Let $d_{\vec{A}}=\prod_{\alpha=1}^L d_{A^\alpha}$.
Combined with Lemma \ref{lemma: lin-zheng} and eigenvalue
of $\mathcal{Y}$ in Lemma \ref{lemma: quotient of knot inv} and
\eqref{computation of q^d(L)},
we have
\begin{align}
&  W_{(\partition{1},\ldots,\partition{1})} (L_{\vec{d}})
    = \sum_{|\vec{A}|=\vec{d}} W_{\vec{A}}(\mathcal{L}) d_{\vec{A}}
    \nonumber
\\
&= \sum_{|\vec{A}|=\vec{d}} d_{\vec{A}}\,
    q^{\sum_\alpha\kappa_{A^\alpha}w(\mathcal{K}_\alpha)/2}
    \cdot t^{\sum_\alpha |A^\alpha| w(\mathcal{K}_\alpha)/2}
    \epsilon_{\vec{A}} \cdot \prod_{\alpha=1}^L
    \dim_q V_{A^\alpha}\,,
    \label{eqn: homfly decomposition}
\end{align}
Where $\epsilon_{\vec{A}}$ is the eigenvalue of
$\mathcal{Y}$ on
$\bigotimes_{\alpha=1}^L V_{A^\alpha}$ as defined in the proof of Lemma
\ref{lemma: quotient of knot inv}. Here if we
change $A^\alpha$ to
$(A^\alpha)^t$, we have
$\kappa_{(A^\alpha)^t}=-\kappa_{A^\alpha}$, which is equivalent to keep
$A^\alpha$ while changing $q$ to $q^{-1}$.

Note that $W_{(\partition{1},\ldots,\partition{1})} (L_{\vec{d}})$ is
essentially a HOMFLY polynomial of $\mathcal{L}_{\vec{d}}$ by Example
\ref{example of quantum group invs}. From the expansion of HOMFLY polynomial
\eqref{eqn: formula of HOMFLY} and Example \ref{example of quantum group invs}
(2), we have
\begin{equation}\label{eqn: symmetry of homfly}
 W_{(\partition{1},\ldots,\partition{1})} (L_{\vec{d}}\,;\,q^{-1},t)
=(-1)^{\sum_\alpha |A_\alpha|} W_{(\partition{1},\ldots,\partition{1})}
(L_{\vec{d}}\,;\,q,t)\,.
\end{equation}

However, one can generalize the definition of quantum group invariants of links
in the following way. Note that in the definition of quantum group invariants,
the enhancement of $\chR$, $K_{2\rho}$, acts on $X_i$ (see
\eqref{eqn: action of K_2rho}) as a scalar $q^{-\frac{1}{2}(N+1-2i)}$. We can
actually generalize this scalar to any $z_i^\alpha$ where $\alpha$ corresponds
the strands belonging to the $\alpha$-th component (cf. \cite{T}). It's not
difficult to see that \eqref{eqn: homfly decomposition} still holds. The quantum
dimension of $V_{A^\alpha}$ thus becomes
$s_{A^\alpha}(z_1^\alpha,\ldots,z_N^\alpha)$ obtained in the same
way as \eqref{eqn: quantum dimension as Schur function}.

We rewrite the above generalized version of quantum group invariants of links
as $W_{\vec{A}}(\mathcal{L};\,q,t;z)$, where $z=\{z^\alpha\}$.
\eqref{eqn: symmetry of homfly} becomes
\begin{equation}\label{eqn: generalized symmetry of homfly}
 W_{(\partition{1},\ldots,\partition{1})} (L_{\vec{d}}\,;\,q^{-1},t;-z)
=(-1)^{\sum_\alpha |A_\alpha|}
    W_{(\partition{1},\ldots,\partition{1})}(L_{\vec{d}}\,;\,q,t;z)
\end{equation}

Now, combine \eqref{eqn: generalized symmetry of homfly},
\eqref{eqn: homfly decomposition} and \eqref{eqn: symmetry of homfly}, we obtain
\begin{align} \label{eqn: equation of symmetry of epsilon}
& \sum_{\vec{A}^t} d_{\vec{A}^t}\,
    q^{-\sum\kappa_{(A^\alpha)^t}w(\mathcal{K}_\alpha)/2}
    \cdot t^{\sum |(A^\alpha)^t| w(\mathcal{K}_\alpha)/2}
    \epsilon_{\vec{A}^t}(q^{-1};\,-z)
     \prod_{\alpha=1}^L s_{(A^\alpha)^t}(-z^\alpha)
    \nonumber
\\
& =(-1)^{\sum_\alpha |A_\alpha|} \cdot \sum_{\vec{A}} d_{\vec{A}}\,
    q^{\sum\kappa_{A^\alpha}w(\mathcal{K}_\alpha)/2}
    \cdot t^{\sum |A^\alpha| w(\mathcal{K}_\alpha)/2}
    \epsilon_{\vec{A}}(q;\,z)
    \prod_{\alpha=1}^L s_{A^\alpha}(z^\alpha)
\end{align}
Note the following facts:
\begin{align}
  s_{A^t}(-z) &=  (-1)^{\ell(A)} s_A (z)\,, \label{eqn: symmetry of Schur
function}
\\
  d_{\vec{A}^t} &= d_{\vec{A}}\,. \label{eqn: d_A equals to d_A-t}
\end{align}
where the second formula follows from
\begin{equation}\label{eqn: chi_A^t of chi_A}
 \chi_{A^t}(C_\mu) = (-1)^{|\mu|-\ell(\mu)}\chi_A(C_\mu)\,.
\end{equation}

Apply \eqref{eqn: symmetry of Schur function} and \eqref{eqn: d_A
equals to d_A-t} to \eqref{eqn: equation of symmetry of epsilon}.
Let $z^\alpha_i=q^{-\frac{1}{2}(N+1-2i)}$, then using $-z$ instead
of $z$ is equivalent to substitute $q$ by $q^{-1}$ while keeping $t$
in the definition of quantum group invariants of links. This can be
seen by comparing
\[
 p_\mu \big(z^\alpha_i=q^{-\frac{N-2i+1}{2}}\big)
=\prod_{j=1}^{\ell(\mu)}
    \frac{t^{-\mu_j/2}-t^{\mu_j/2}}{q^{-\mu_j/2}-q^{\mu_j/2}}
\]
with
\[
 p_\mu (-z^\alpha) = (-1)^{\ell(\mu)} p_\mu (z^\alpha)\,.
\]
Therefore, we have
\begin{equation} \label{eqn: symmetry of epsilon}
 \epsilon_{\vec{A}^t}(q^{-1},t)
=\epsilon_{\vec{A}}(q,t) \,.
\end{equation}

By the formula of quantum dimension, it is easy to verify that
\begin{equation} \label{eqn: symmetry of trivial link}
 W_{\vec{A}^{t}}(\bigcirc^L;\,q,t)
=(-1)^{\parallel\vec{A}\parallel}W_{\vec{A}}(\bigcirc^L;\,q^{-1},t)
\end{equation}
Combining \eqref{eqn: eigenvalue of Y with quantum dimension},
\eqref{eqn: symmetry of epsilon} and
\eqref{eqn: symmetry of trivial link},
the proof of the Lemma is then completed.
\end{proof}

By Lemma \ref{lemma: symmetry}, we have the following expansion
about $P_{\vec{B}}$
\[
 P_{\vec{B}}(q,t)
=\sum_{g\geq 0}\sum_{Q\in \mathbb{Z}/2}N_{\vec{B},g,Q}
 (q^{-1/2}-q^{1/2})^{2g-2N_{0}}t^{Q}
\]
for some $N_{0}$. We will show $N_{0}\leq 1$.

Let $q=e^{u}$. The pole order
of $( q^{-1/2}-q^{1/2}) $ in $P_{\vec{B}}$ is the same as pole
order of $u$.

Let $f(u)$ be a Laurent series in $u$. Denote $\deg _{u}f$ to be the lowest
degree of $u$ in the expansion of $u$ in $f$.

Combined with (\ref{eqn: F_mu of P_B}), $N_{0}\leq 1$ follows from the
following lemma.

\begin{lemma}\label{lemma: degree}
$\deg _{u}F_{\vec{\mu}}\geq \ell(\vec{\mu})-2$.
\end{lemma}

Lemma \ref{lemma: degree} can be proved through the following
cut-and-join analysis.

\subsection{Cut-and-join analysis}

\subsubsection{Cut-and-join operators}

Let $\tau=(\tau_1,\cdots, \tau_L)$, substitute
\[
 W_{\vec{A}}(\mathcal{L};\,q,t;\tau)
=W_{\vec{A}}(\mathcal{L};\,q,t)\cdot
    q^{\sum_{\alpha=1}^{L}\kappa_{A^\alpha}\tau_\alpha/2}
\]
in the Chern-Simons partition function, we have the following framed partition
function
\[
 Z( \mathcal{L};\,q,t,\tau)
=1+\sum_{\vec{A}\neq 0}W_{\vec{A}}(\mathcal{L};\,q,t,\tau )\cdot
    s_{\vec{A}}(x) \,.
\]
Similarly, framed free energy
\[
 F(\mathcal{L};\,q,t,\tau)
= \log Z( \mathcal{L};\,q,t,\tau)\,.
\]
We also defined framed version of $Z_{\vec{\mu}}$ and $F_{\vec{\mu}}$ as
follows
\begin{align*}
& Z(\mathcal{L};\,q,t,\tau )
  =1+\sum_{\vec{\mu}\neq 0}Z_{\vec{\mu}}(q,t,\tau)\cdot
    p_{\vec{\mu}}(x) \,,
\\
& F(\mathcal{L};\,q,t,\tau )
  =\sum_{\vec{\mu}\neq 0}F_{\vec{\mu}}(q,t,\tau)p_{\vec{\mu}}(x)
    \,.
\end{align*}
One important fact of these framing series is that they satisfy the
cut-and-join equation which will give a good control of $F_{\vec{\mu}}$.

Define exponential cut-and-join operator $\mathfrak{E}$
\begin{equation}\label{def: exp cut-and-join operator}
  \mathfrak{E}
= \sum_{i,j\geq 1}
    \bigg(
     ijp_{i+j}\frac{\partial ^{2}}{\partial p_{i}\partial p_{j}}
    +(i+j)p_{i}p_{j}
     \frac{\partial }{\partial p_{i+j}}
    \bigg)  \,,
\end{equation}
and log cut-and-join operator $\mathfrak{L}$
\begin{equation}\label{def: log cut-and-join operator}
  \mathfrak{L}F
= \sum_{i,j\geq 1}
    \bigg(
      ijp_{i+j}\frac{\partial^{2}F}{\partial p_{i}\partial p_{j}}
     +(i+j)p_{i}p_{j}
        \frac{\partial F}{\partial p_{i+j}}
     +ijp_{i+j}
      \frac{\partial F}{\partial p_{i}}
        \frac{\partial F}{\partial p_{j}}
    \bigg) \,.
\end{equation}
Here $\{ p_i\} $ are regarded as independent variables. Schur
function $s_{A}(x)$ is then a function of $\{p_{i}\}$.

Schur function $s_{A}$ is an eigenfunction of exponential cut-and-join with
eigenvalue $\kappa _{A}$ \cite{FW,G,Z1}. Therefore,
$Z( \mathcal{L};\,q,t,\tau)$ satisfies the following exponential
cut-and-join equation
\begin{align}\label{eqn: linear cut-and-join}
 \frac{\partial Z(\mathcal{L};\,q,t,\tau )}{\partial \tau_\alpha }
=\frac{u}{2} \mathfrak{E}_{\alpha} Z(\mathcal{L};\,q,t,\tau ) \,,
\end{align}
or equivalently, we also have the log cut-and-join equation
\begin{equation}\label{eqn: non-linear cut-and-join}
 \frac{\partial F(\mathcal{L};\,q,t,\tau )}{\partial \tau_\alpha }
=\frac{u}{2} \mathfrak{L}_{\alpha} F(\mathcal{L};\,q,t,\tau )\,.
\end{equation}
In the above notation, $\mathfrak{E}_{\alpha }$ and $\mathfrak{L}_{\alpha }$
correspond to variables $\left\{ p_{i}\right\} $ which take value of
$\left\{ p_{i}(x^{\alpha })\right\} $.

\eqref{eqn: non-linear cut-and-join} restricts to $\vec{\mu}$ will
be of the following form:
\begin{equation}\label{eqn: nonlinear terms}
  \frac{\partial F_{\vec{\mu}}}{\partial \tau_\alpha }
= \frac{u}{2}
     \bigg(
    \sum_{| \vec{\nu}| =| \vec{\mu}| ,\,\ell(\vec{\nu})=\ell(\vec{\mu})\pm
1}
      \alpha _{\vec{\mu}\vec{\nu}}F_{\vec{\nu}}
    + \textrm{nonlinear terms}
     \bigg) \,,
\end{equation}
where $\alpha _{\vec{\mu}\vec{\nu}}$ is some constant, $\vec{\nu}$ is obtained
by cutting or jointing of $\vec{\mu}$.

Recall that given two partitions $A$ and $B$, we say $A$ is a
cutting of $B$ if one cuts a row of the Young diagram of $B$ into
two rows and reform it into a new Young diagram which happens to be
the Young diagram of $A$, and we also say $B$ is a joining of $A$.
For example, $(7,3,1)$ is a joining of $(5,3,2,1)$ where we join $5$
and $2$ to get $7$ boxes. Using Young diagram, it looks like
\begin{align*}
 \partition{5 3 2 1} \quad
\stackrel{\textrm{\tiny join}}{\Longrightarrow} \quad
 \partition{7 3 1} \quad
\stackrel{\textrm{\tiny cut}}{\Longrightarrow} \quad
 \partition{5 3 2 1} \,.
\end{align*}
In \eqref{eqn: nonlinear terms}, cutting and joining happens only
for the $\alpha$-th partition.

\subsubsection{Degree of $u$}

By \eqref{eqn: F_mu of Z_mu}, it is easy to see through induction
that for two links $\mathcal{L}_1$ and $\mathcal{L}_2$,
\begin{align}\label{eqn: connected homfly of disjoint links}
 F_{(\partition{1},\cdots,\partition{1})} (\mathcal{L}_1 \otimes
    \mathcal{L}_2)
=0\,.
\end{align}
For simplicity of writing, we denote by
\[
 F_{(\partition{1},\cdots,\partition{1})} (\mathcal{L})
=F^\circ(\mathcal{L})
\]

Note that when we put the labeling irreducible representation by the
fundamental ones, quantum group invariants of links reduce to HOMFLY
polynomials except for a universal factor. Therefore, if we apply
skein relation, the following version of skein relation can be
obtained. Let positive crossing
\[
 \mathcenter{\includegraphics[height=0.8cm]{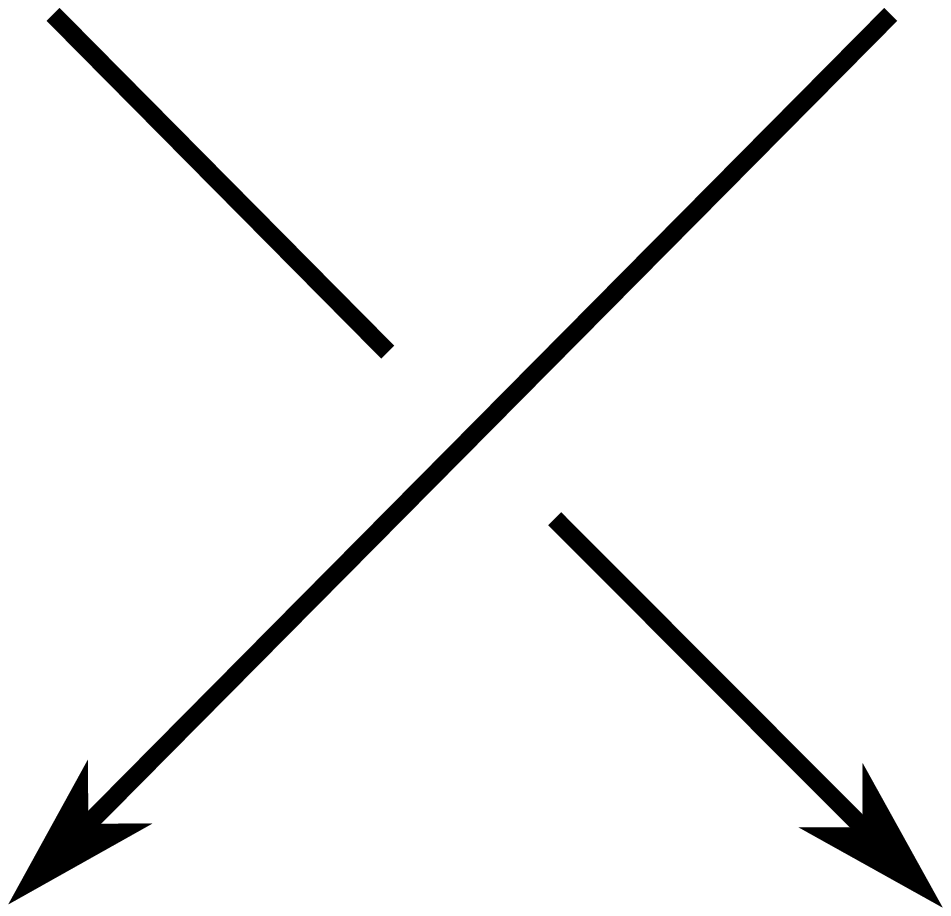}}
\]
used later appear between two different components, $\mathcal{K}_1$
and $\mathcal{K}_2$, of the link. Denote by $lk=lk(\mathcal{K}_1,
\mathcal{K}_2)$, then:
\begin{align}\label{eqn: skein relation for F}
 F^\circ\Big(\mathcenter{\includegraphics[height=0.8cm]{L+.eps}}\Big) -
    F^\circ\Big(\mathcenter{\includegraphics[height=0.8cm]{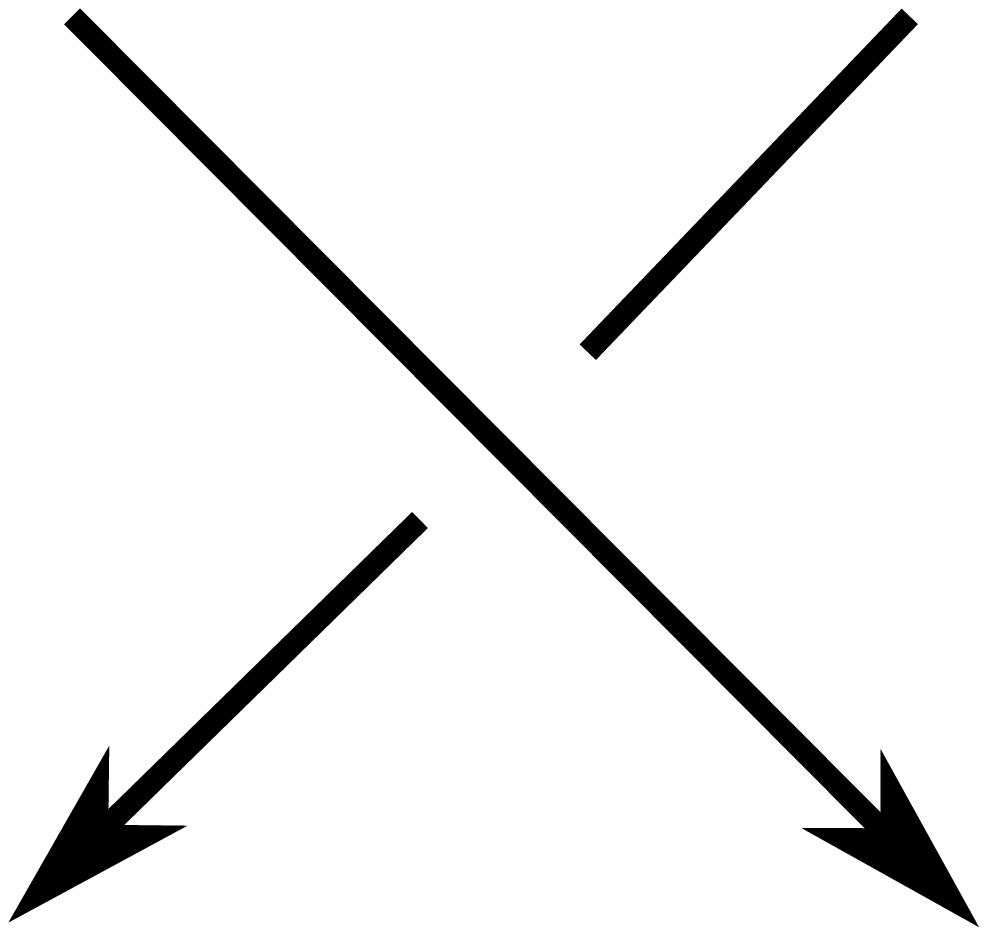}}\Big)
=t^{-lk+\frac{1}{2}} (q^{-\frac12}-q^{\frac12})
    F^\circ\Big(\mathcenter{\includegraphics[height=0.8cm]{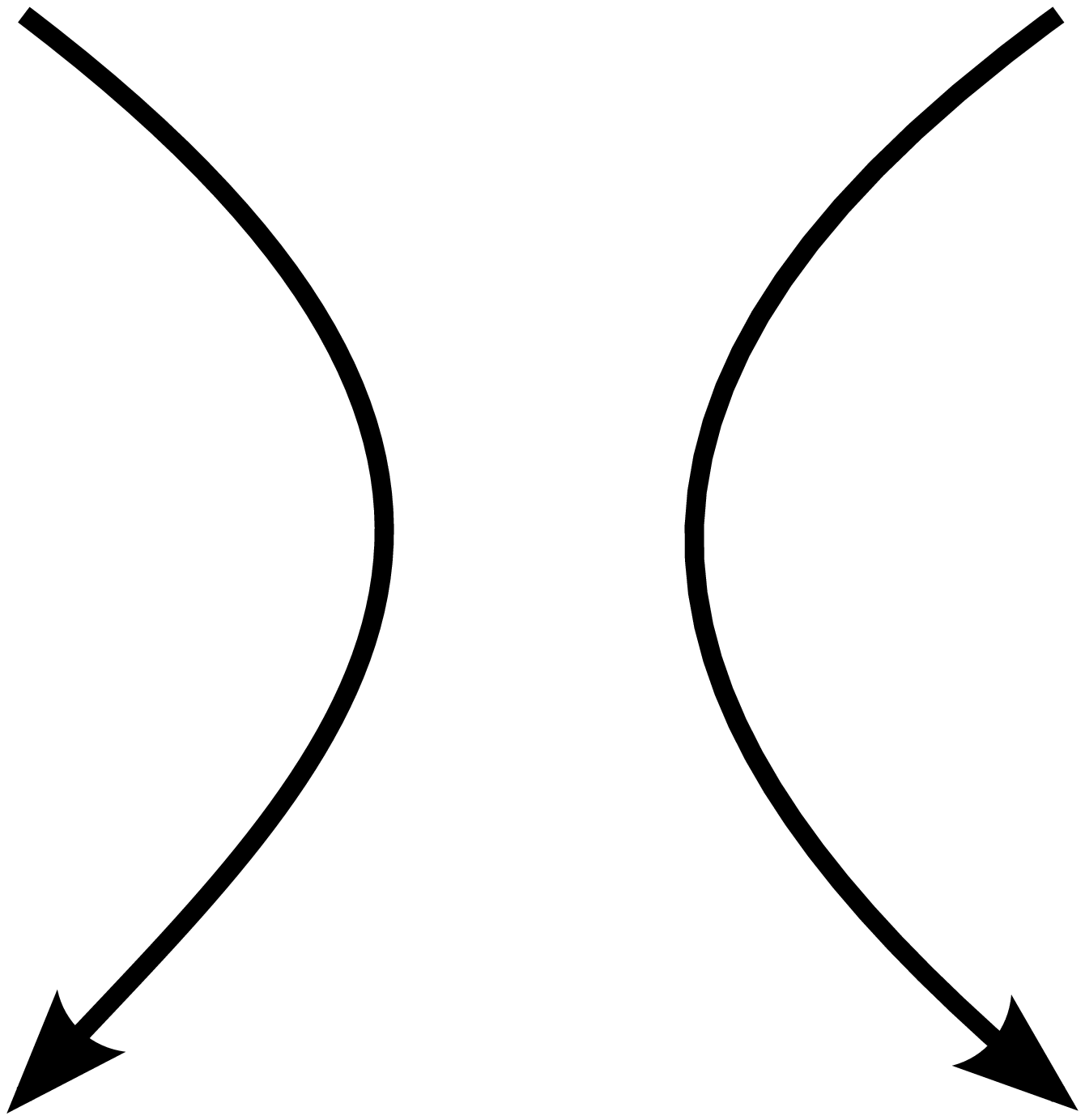}}\Big)
\end{align}

We want to claim that
\begin{align}\label{eqn: deg_u of F^circ}
 \deg_u F^\circ (\mathcal{L}) \geq L-2\,,
\end{align}
where $L$ is the number of components of $\mathcal{L}$.

Firstly, for a knot $\mathcal{K}$, a simple computation shows that:
\[
 \deg_u F_{\partition{1}} (\mathcal{K}) =-1\,.
\]

Using induction, we may assume that claim \eqref{eqn: deg_u of
F^circ} holds for $L\leq k$. When $L=k+1$, by \eqref{eqn: skein
relation for F},
\begin{align}
 \deg_u \Big(F^\circ\Big(\mathcenter{\includegraphics[height=0.8cm]{L+.eps}}\Big) -
    F^\circ\Big(\mathcenter{\includegraphics[height=0.8cm]{L-.eps}}\Big)
    \Big)
=1+\deg_u
    \Big( F^\circ\Big(\mathcenter{\includegraphics[height=0.8cm]{L0.eps}}\Big)
    \Big)
\geq k-1 \,.
\end{align}
However, if
\[
 \deg_u
    \Big(F^\circ\Big(\mathcenter{\includegraphics[height=0.8cm]{L+.eps}}\Big)
    \Big)
< k-1\,,
\]
this will imply that in the procedure of unlinking $\mathcal{L}$,
the lowest degree term of $u$ in $F^\circ$ are always the same.
However, this unlinking will lead to $F^\circ$ equal to $0$ due to
\eqref{eqn: connected homfly of disjoint links}, which is a
contradiction! This implies that if the number of components of
$\mathcal{L}_1$ is greater than the number of components of
$\mathcal{L}_2$, we always have
\begin{align*}
 \deg_u F^\circ(\mathcal{L}_1) > \deg_u F^\circ (\mathcal{L}_2)
\end{align*}

Therefore, we proved the claim \eqref{eqn: deg_u of F^circ}.

\subsubsection{Induction procedure of cut-and-join analysis}

Let
$\mu^i = (\mu^i_1, \cdots, \mu^i_{\ell_i})$.
We use symbol
\begin{align*}
    \hat{Z}_{ \vec{\mu} } =  Z_{ \vec{\mu} } \cdot \mathfrak{z}_{\vec{\mu}} ; \qquad
        \hat{F} = F_{\vec{\nu}} \cdot \mathfrak{z}_{ \vec{\mu} }\,.
\end{align*}
Notice the following fact from the definition of quantum group
invariants:
\begin{align}
    \hat{Z}_{(\mu^1,\cdots, \mu^L)} (\mathcal{L})
    &=\hat{Z}_{(\mu^1_1), \cdots, (\mu^1_{\ell_1}), \cdots, (\mu^L_1),\cdots,
        (\mu^L_{\ell_L})} (\mathcal{L}_{\vec{\mu}}) \,,
        \label{eqn: colored homfly and calbing}
\\
 \hat{F}_{(\mu^1,\cdots, \mu^L)} (\mathcal{L})
    &=\hat{F}_{(\mu^1_1), \cdots, (\mu^1_{\ell_1}), \cdots, (\mu^L_1),\cdots,
        (\mu^L_{\ell_L})} (\mathcal{L}_{\vec{\mu}})
        \label{eqn: connected colored homfly and calbing} \,.
\end{align}

Let $\tau=(\tau_1, 0,\cdots,0)$. Similar as \eqref{eqn: convolution
for G_mu in MV-formula}, we give the following formula:
\begin{align}
  F_{\vec{\mu}}
&=\sum_{|\Lambda|=\vec{\mu}} \Theta_\Lambda
    \prod_{\alpha=1}^{\ell(\Lambda)}
    \sum_{A^\alpha} \frac{\chi_{A^\alpha}(\Lambda^\alpha)}{z_{\Lambda^\alpha}}
    W_{\vec{A}}(q,t) q^{\frac{\kappa_{A^1}\tau_1}{2}}
        \nonumber
\\
&=F_{\vec{\mu}}(0) + \sum_{p\geq 1} \frac{(\frac{u\tau_1}{2})^p}{p!}
    \sum_{|\Lambda|=\vec{\mu}} \Theta_\Lambda \sum_{\Omega\neq (1^\Lambda),
    \mathcal{A}} \frac{\chi_{\mathcal{A}}(\Lambda)}{z_\Lambda}
    \frac{\chi_{\mathcal{A}}(\Omega)}{z_\Omega} Z_\Omega(0)
    \kappa_{\mathcal{A}}^p
        \nonumber
\\
& \qquad
    +Z_{1^\Lambda}(0) \sum_{p\geq 1}
    \frac{(\frac{u\tau_1}{2})^p}{p!} \sum_{|\Lambda|=\vec{\mu}} \Theta_\Lambda
    \sum_{\mathcal{A}}
    \frac{\chi_{\mathcal{A}}(\Lambda)}{z_{\Lambda}}
    \frac{\chi_{\mathcal{A}}(1^\Lambda)}{z_{1^\Lambda}} \kappa_A^p
        \label{eqn: leading term in F_mu}
\end{align}

Let $g_{\vec{\mu}}^{(\alpha)}(\tau_\alpha)$ be the degree of
$\tau_\alpha$ in the coefficient of the lowest degree of $u$ in
$F_{\vec{\mu}}$. By \eqref{ineq: deg_u F_mu = L_mu-2 in MV-formula}
and \eqref{eqn: leading term in F_mu}, if $|\mu_\alpha|>1$, we have
\[
 \deg_{\tau_\alpha} g_{\vec{\mu}}^{(\alpha)}
=|\mu^\alpha|+\ell(\mu^\alpha) -2 >0 \,.
\]

By induction, assume that if $|\mu^\alpha|\leq d$, $\deg_u
F_{\vec{\mu}} \geq \ell(\vec{\mu})-2$ and hence by \eqref{ineq:
deg_u F_mu = L_mu-2 in MV-formula}, $\deg_{\tau_\alpha}
g_{\vec{\mu}}^{(\alpha)}(\tau_\alpha) =|\mu^\alpha|+\ell(\mu^\alpha)-2$.
Combined with \eqref{eqn: connected colored homfly and calbing}, if
$|\lambda^\alpha|=d+1$ and $\ell(\lambda^\alpha)>1$, we have:
\begin{align*}
 \deg_u F_{\vec{\lambda}}=\ell(\vec{\lambda})-2\,.
\end{align*}
Without loss of generality, assume $\mu^1=(d,1)$, we will just
consider the cut-and-join equation for $\tau_1$:
\begin{align*}
 \frac{\partial F_{(\mu^1, \cdots, \mu^L)}} {\partial \tau_1}
=\frac{u}{2} \Big( d F_{\big((d+1), \mu^2 \cdots, \mu^L\big)} +
    \sum_{\ell(\vec{\nu})=\ell(\vec{\mu})+1}
    \beta_{\vec{\nu}} F_{\vec{\nu}} \Big) +\ast
\end{align*}
where $\beta_{\vec{\nu}}$ are some constants and $\ast$ represents
some non-linear terms in the cut-and-join equation. The crucial
observation of this non-linear terms is that its degree in $u$ is
equal to $\ell(\vec{\mu})-2$. Comparing the degree in $u$ on both
sides of the equation, we have:
\begin{align*}
 \deg_u F_{\big((d+1), \mu^2, \cdots, \mu^L\big)}
=1+\sum_{\alpha=2}^L \ell(\mu^\alpha) -2\,.
\end{align*}
This completes the induction. The proof of Lemma \ref{lemma: degree}
follows immediately.

    Define
        \begin{align*}
            \widetilde{F}_{\vec{\mu}} = \frac{F_{\vec{\mu}} } {\phi_{ \vec{\mu} }(q)}, \qquad \widetilde{Z}_{ \vec{\mu} }
                = \frac{ Z_{ \vec{\mu} } }{ \phi_{ \vec{\mu} }(q) }\,.
        \end{align*}
    Lemma \ref{lemma: degree} directly implies follows:
    \begin{corollary} \label{cor: tilde_F_mu}
        $\widetilde{F}$ are of the following form:
        \begin{align*}
            \widetilde{F}_{\vec{\mu}}(q,t)
            = \sum_{\textrm{finitely many } n_\alpha} \frac{ a_\alpha(t) }
                { [n_\alpha]^2 } + \textrm{ polynomial.}
        \end{align*}
    \end{corollary}

    \begin{remark}
        Combine the above Corollary, \eqref{eqn: connected homfly expansion}
        and \eqref{eqn: F_mu of Z_mu}, we have:
        \begin{align}\label{eqn: F of homfly expansion}
            \widetilde{F}_{(\partition{1},\ldots,\partition{1})}(q,t)
            = \frac{ a(t) }
                { [1]^2 } + \textrm{ polynomial.}
        \end{align}
    \end{remark}


 \subsection{Framing and pole structures}
    Consider $\delta_n = \sigma_1\cdots \sigma_{n-1}$.
    Let $\fS_A$ be the minimal projection from $\mathcal{H}_n \rightarrow \mathcal{H}_A$, and
        \begin{align*}
            \fP_\mu = \sum_A \chi_A(C_\mu) \fS_A
        \end{align*}
    We will apply a lemma of Aiston-Morton \cite{AM} in the following computation:
        \begin{align*}
            \delta_n^n \fS_A = q^{\frac12 \kappa_A} \fS_A\,.
        \end{align*}
        Let
        \begin{align*}
            \vec{d} = \big( (d_1),\ldots, (d_L) \big), \qquad \frac{1}{\vec{d}} = \Big( \frac{1}{d_1}, \ldots, \frac{1}{d_L} \Big)\,.
        \end{align*}
        Due to the cabling formula to the length of partition
        \eqref{eqn: colored homfly and calbing} and
        \eqref{eqn: connected colored homfly and calbing},
        we can simply deal with all the color of one row without loss of generality.
        Take framing $\tau_\alpha=n_\alpha + \frac{1}{d_\alpha}$ and choose a braid group representative of $\cL$ such that the writhe number of $\cL_\alpha$ is $n_\alpha$. Denote by $\vec{\tau} = (\tau_1,\ldots,\tau_L)$,
        \begin{align*}
            \hZ_{\vec{d}} \Big(\cL;q,t; \vec{\tau} \Big)
                &= \sum_{\vec{A} } \chi_{\vec{A}} (C_{\vec{d}} )W_{\vec{A}}(\cL;q,t)
                    q^{\frac12 \sum_{\alpha=1}^L
                    \kappa_{A^\alpha} \big(n_\alpha + \frac{1}{ d_\alpha}\big) } \\
                &= t^{-\frac12 \sum_\alpha d_\alpha n_\alpha}
                    \Tr \Big( \cL_{\vec{d}}
                    \sum_{\vec{A}} \chi_{\vec{A}}(C_{\vec{d}}) q^{\frac12 \sum_{\alpha} \kappa_\alpha \frac{1}{ d_\alpha} }
                    \otimes_{\alpha=1}^L\fS_{A^\alpha} \Big) \\
                &= t^{-\frac12\sum_\alpha d_\alpha n_\alpha}
                    \Tr\Big( \cL_{\vec{d}}
                    \sum_{\vec{A}} \chi_{\vec{A}} (C_{\vec{d}}) (\delta_{d_1}\otimes\cdots\otimes \delta_{d_L})
                    \otimes_{\alpha=1}^L\fS_{A^\alpha} \Big)\\
                &= t^{-\frac12\sum_\alpha d_\alpha n_\alpha}
                    \Tr \Big( \cL_{\vec{d}} \cdot \otimes_{\alpha=1}^L\delta_\alpha \cdot
                    \fP_{(1)}^{(d_1)} \otimes \cdots \otimes \fP_{(1)}^{(d_L)} \Big) \,.
        \end{align*}
        Here, $\fP_{(1)}^{(d_\alpha)}$ means that in the projection, we use $q^{d_\alpha}, t^{d_\alpha}$ instead of using $q,t$.
        If we denote by
        \begin{align*}
            \cL\ast Q_{\vec{d}} = \cL_{\vec{d}} \cdot \delta_{\vec{d}} \cdot
                    \fP_{(1)}^{(d_1)} \otimes \cdots \otimes \fP_{(1)}^{(d_L)} \,,
        \end{align*}
        we have
        \begin{align}\label{eqn: Z_mu as homfly of twist}
            \hZ_{\vec{d}} \Big(\cL;q,t; \frac{1}{\vec{d}} \Big) = \mathfrak{H}(\cL\ast Q_{\vec{d}})\,.
        \end{align}
        Here $\mathfrak{H}$ is the HOMFLY polynomial which is normalized as
        \begin{align*}
            \mathfrak{H} (\mathrm{unknot}) = \frac{ t^{\frac12} - t^{-\frac12} }{ q^{\frac12} -q^{-\frac12} }\,.
        \end{align*}
        With the above normalization, for any given link $\cL$, we have
        \begin{align*}
            [1]^L \cdot\mathfrak{H} (\cL) \in \mathbb{Q} \big[ [1]^2, t^{\pm\frac12} \big]\,.
        \end{align*}
        Substituting $q$ by $q^{d_\alpha}$ in the corresponding component, it leads to the follows:
        \begin{align}\label{eqn: pole of Z_mu}
            \prod_{\alpha=1}^L [d_\alpha] \cdot
            \hat{Z}_{\vec{d}} (\cL; q,t; \vec{\tau})
            \in \mathbb{Q}\big[ [1]^2, t^{\pm\frac12} \big]\,.
        \end{align}

        On the other hand, given any two frames $\vec{\omega}=(\omega_1,\ldots,\omega_L)$,
        \begin{align*}
            \hat{Z}_{\vec{\mu}} (\cL;\, q,t; \vec{\omega})
            &= \sum_{\vec{A}} \chi_{\vec{A}} (C_{\vec{\mu}})
                W_{\vec{A}}(\cL;\,q,t;\vec{\omega}) \\
            &= \sum_{\vec{A}} \chi_{\vec{A}}(C_{\vec{\mu}})
                \sum_{\vec{\nu}} \frac{ \chi_{\vec{A}}(C_{\vec{\nu}})} {\mathfrak{z}_{\vec{\nu}}} \hat{Z}_{\vec{\nu}} (\cL;\,q,t)
                q^{\frac12\sum_\alpha \kappa_{A^\alpha}\omega_\alpha }\,.
        \end{align*}
        Exchange the order of summation, we have the following {\em convolution formula}:
        \begin{align}\label{eqn: convolution formula}
            \hat{Z}_{\vec{\mu}}(\cL;\,q,t;\vec{\omega})
            = \sum_{\vec{\nu}} \frac{ \hat{Z}_{\vec{\nu}}(\cL;\,q,t)}
                {\mathfrak{z}_{\vec{\nu}}}
                \sum_{\vec{A}} \chi_{\vec{A}}(C_{\vec{\mu}})
                \chi_{\vec{A}}(C_{\vec{\nu}})
                q^{\frac12\sum_\alpha \kappa_{A^\alpha} \omega_\alpha}\,.
        \end{align}

        Return to \eqref{eqn: pole of Z_mu}. This property holds for arbitrary choice of $n_\alpha$, $\alpha=1,\ldots,L$. By convolution formula, the coefficients of possible other poles vanish for arbitrary integer $n_\alpha$. $q^{n_\alpha}$ is involved through certain polynomial relation, which implies the coefficients for other possible poles are simply zero. Therefore,
        \eqref{eqn: pole of Z_mu} holds for any frame.

        Consider $\hat{Z}_{\vec{d}}(\cL)$, where each component of $\cL$ is labeled by a young diagram of one row. As discussed
        above, this assumption does not lose any generality.
        Let $\mathcal{K}$ be a component of $\cL$ labeled by the color $(c)$.
        If we multiply $\widetilde{Z}_{\vec{d}}(\cL)$ by $[c]^2$, we call it normalizing $\widetilde{Z}_{\vec{d}}(\cL)$ w.r.t $\mathcal{K}$.
        Similarly, we call $[c]^2\widetilde{F}_{\vec{d}}(\cL)$ normalizing $\widetilde{F}_{\vec{d}}(\cL)$ w.r.t $\mathcal{K}$.

        By \eqref{eqn: F of homfly expansion}, after normalizing $\widetilde{F}_{(\partition{1},\ldots,\partition{1})}(\cL)$ w.r.t
        any component of $\cL$, one will obtain a polynomial in terms of $[1]^2$ and $t^{\pm\frac12}$.
        Note that
        \begin{align*}
            \widetilde{Z}_{\vec{d}} = \sum_{\Lambda\vdash \vec{d}}
                \frac{ \widetilde{F}_{\Lambda} }{ \Aut |\Lambda| }\,.
        \end{align*}
        Therefore, if we normalizing on both side w.r.t $\mathcal{K}$, the equality still holds.
        Applying \eqref{eqn: Z_mu as homfly of twist}, one has
        \begin{align}\label{eqn: normalizing F is polynomial}
            [c]^2 \widetilde{F}_{\vec{d}}(\cL;q,t) \in \mathbb{Q}\big[ [1]^2, t^{\pm\frac12} \big]\,.
        \end{align}
        We thus proved the following proposition:

        \begin{proposition} \label{prop: Z_mu pole structure}
            Notation as above, we have:
        \begin{align}
            \prod_{\alpha=1}^L [d_\alpha] \cdot
                \hZ_{\vec{d}} (\cL;q,t)
            & \in \mathbb{Q} [ [1]^2, t^{\pm\frac12}]; \\
                [d_\alpha]^2 \widetilde{F}_{\vec{d}}
                ( \cL; q,t )
            &\in \mathbb{Q}\big[ [1]^2, t^{\pm\frac12} \big], \forall \alpha\,.
        \end{align}
        \end{proposition}

        For a given $\vec{d}$, denote by $D_{\vec{d}}$ the
        $\gcd\{d_1, \ldots, d_L\}$.
        Proposition \ref{prop: Z_mu pole structure} implies that the principle part of $\widetilde{F}$ are possible summations of $\frac{a(t)}{[k]^2}$, where $k$ divides all $d_\alpha$'s, or equivalently, $k|D_{\vec{d}}$.

        Similar as above, choose a braid group element representative for $\cL$ such that its $\alpha$-th component has writhe number $n_\alpha$ and choose frame as $\vec{\tau}=(n_1+\frac{1}{D_{\vec{d}} },\ldots, n_L+\frac{1}{D_{\vec{d}}})$.
        \begin{align}
            \hZ_{\vec{d}}(\cL;q,t; \vec{\tau})
            &= \sum_{\vec{A}} \chi_{\vec{A}}(C_{\vec{d}})  W_{\vec{A}}(\cL;q,t)
                q^{ \frac12\sum_\alpha \kappa_{A^\alpha} (n_\alpha+\frac{1}{D_{\vec{d}}} ) } \nonumber \\
            &= t^{-\frac12\sum_\alpha d_\alpha n_\alpha}
                \Tr \Big(  \cL_{\vec{d}} \sum_{\vec{A}} \chi_{\vec{A}}(C_{\vec{d}}) q^{ \frac12\sum_\alpha \kappa_{A^\alpha} \frac{1}{D_{\vec{d}}} }
                \otimes_{\alpha=1}^L\fS_{A^\alpha} \Big) \nonumber \\
            &= t^{-\frac12\sum_\alpha d_\alpha n_\alpha}
                \Tr (\cL_{\vec{d}} \cdot \otimes_{\alpha=1}^L\delta_{d_\alpha}^{\frac{d_\alpha}{D_{\vec{d}}} } \otimes_{\alpha=1}^L
                \fP_{(d_\alpha/D_{\vec{d}})}^{(D_{\vec{d}})} )
                \nonumber \\
            &= t^{-\frac12\sum_\alpha d_\alpha n_\alpha}
                \hZ_{\vec{d}/D_{\vec{d}}} (\cL_{\vec{d}} \cdot \otimes_{\alpha=1}^L
                \delta_{d_\alpha}^{\frac{d_\alpha}{D_{\vec{d}}} } ; q^{D_\mu}, t^{D_\mu} ) \,.
                \label{eqn: frame of D_mu}
        \end{align}
        This implies that $\hZ_{\vec{d}}$ is a rational function of
        $q^{\pm\frac12 D_{\vec{d}}}$ and
        $t^{\pm \frac12 D_{\vec{d}}}$.

        Passing this to $\widetilde{F}_{\vec{d}}$, we have $\widetilde{F}_{\vec{d}}$
        is a rational function of $[D_{\vec{d}}]^2$ and $t^{\pm\frac12 D_{\vec{d}}}$.
        As discussed above, the principle part of $\widetilde{F}_{\vec{d}}$ are summations
        of $\frac{ a(t) }{ [k]^2 }$ where $k|D_{\vec{d}}$. Combining them, we have
        \begin{align} \label{eqn: pole of tilde_F_mu}
            \widetilde{F}_{\vec{d}}(\cL;q,t)
            = \frac{ H_{\vec{d}/D_{\vec{d}}} ( t^{D_{\vec{d}}} ) } { [D_{\vec{d}}]^2} +
            \textrm{polynomial in $[D_{\vec{d}}]^2$ and $t^{\pm\frac12 D_{\vec{d}}}$}\,.
        \end{align}
        Once again, due to arbitrary choice of $n_\alpha$, we know the above pole structure of $\widetilde{F}_{\vec{d}}$ holds
        for any frame.

        Now we will show $H_{\vec{d}/D_{\vec{d}}} (t)$ only depends on $\vec{d}/D_{\vec{d}}$ and $\cL$. If one checks
        \eqref{eqn: frame of D_mu}, one finds the computation also follows if $1/D_{\vec{d}}$ is replaced by $1/k$ for any
        $k| D_{\vec{d}}$. By induction, we obtain that $H_{\vec{d}/D_{\vec{d}}} (t)$ only depends on $\vec{d}/D_{\vec{d}}$ and
        $\cL$.

        Combining \eqref{eqn: connected colored homfly and calbing}, we proved the following Proposition:
        \begin{proposition} \label{prop: pole structure of tilde_F_mu}
            Notations are as above. Assume $\cL$ is labeled by the color $\vec{\mu}=(\mu^1,\ldots,\mu^L)$. Denote by $D_{\vec{\mu}}$ is the greatest
            common divisor of $\{\mu^1_1,\ldots,\mu^1_{\ell(\mu^1)},\ldots, \mu^i_j,\ldots, \mu^L_{\ell(\mu^L)}\}$.
            $\widetilde{F}_{\vec{\mu}}$ has the following structure:
            \begin{align*}
                \widetilde{F}_{ \vec{\mu} } (q,t)  = \frac{ H_{\vec{\mu}/D_{\vec{\mu}}}( t^{D_{\vec{\mu}}} ) } { [D_{\vec{\mu}} ]^2 } +
                f(q,t)\,,
            \end{align*}
            where $H_{\vec{\mu}/D_{\vec{\mu}}}(t)$ only depends on $\vec{\mu}/D_{\vec{\mu}}$ and $\cL$, $f(q,t)\in\mathbb{Q}\big[ [1]^2, t^{\pm\frac12} \big]$.
        \end{proposition}

        \begin{remark}
            In Proposition \ref{prop: pole structure of tilde_F_mu}, it is very interesting to interpret in topological string side that
            $H_{\vec{\mu}/D_{\vec{\mu}}}(t)$ only depends on $\vec{\mu}/D_{\vec{\mu}}$ and $\cL$. The principle term is generated due to
            summation of counting rational curves and independent choice of $k$ in the labeling color
            $k\cdot\vec{\mu}/D_{\vec{\mu}}$. This phenomenon simply tells us that contributions of counting rational curves in
            the labeling color $k\cdot\vec{\mu}/D_{\vec{\mu}}$ are through multiple cover contributions of
            $\vec{\mu}/D_{\vec{\mu}}$.
        \end{remark}


\section{Proof of Theorem \ref{thm: integrality}} \label{sec: integrality}

\subsection{A ring characterizes the partition function}\label{subsec: Integrality - ring structure}

Let $\vec{d}=(d_{1},...,d_{L})$, $y=(y^1,\ldots,y^L)$. Define
\begin{equation}\label{eqn: def of T_d}
 T_{\vec{d}}
=\sum_{\vec{B}=\vec{d}} s_{\vec{B}}(y) P_{\vec{B}}(q,t) \,.
\end{equation}
By the calculation in Appendix \ref{appendix: T_vd calculation}, one will get
\begin{equation}\label{eqn: integrality eqn}
 T_{\vec{d}}
=
 q^{| \vec{d}| }
    \sum_{k|\vec{d}}\frac{\mu (k)}{k}
    \sum_{\mathfrak{A}\in \mathcal{P(P}^{n}),
        \parallel\mathfrak{A}\parallel =\vec{d}/k}
    \theta_{\mathfrak{A}}W_{\mathfrak{A}}(q^{k},t^{k})
    s_{\mathfrak{A}}(z^{k})
\end{equation}
However,
\begin{align*}
   T_{\vec{d}}
&=\sum_{| \vec{B}| =\vec{d}}s_{\vec{B}}(y)
    P_{\vec{B}}(q,t)
\\
&=\sum_{g= 0}^\infty \sum_{Q\in \mathbb{Z}/2}
    \bigg( \sum_{|\vec{B}| =\vec{d}}N_{\vec{B};\,g,Q}s_{\vec{B}}(y)
    \bigg)
    (q^{1/2}-q^{-1/2})^{2g-2}t^{Q} \,.
\end{align*}
Denote by $\Omega(y)$ the space of all integer coefficient symmetric functions
in $y$. Since Schur functions forms a basis of $\Omega(y)$ over $\mathbb{Z}$,
$N_{\vec{B};\,g,Q}\in \mathbb{Z}$ will follow from
\[
 \sum_{|\vec{B}|
=\vec{d}}N_{\vec{B};\,g,Q}s_{\vec{B}}(y)\in \Omega(y) \, .
\]

Let $v=[1]^2_q$. It's easy to see that $[n]^2_q$ is a monic polynomial of $v$
with integer coefficients. The following ring is very crucial in
characterizing the algebraic structure of Chern-Simons partition function, which
will lead to the integrality of $N_{\vec{B};\,g,Q}$.
\[
   \mathfrak{R}(y;v,t)
=  \bigg\{ \frac{a(y;v,t)}{b(v)}:\, a(y;v,t)\in
    \Omega(y)[v,t^{\pm 1/2}],
    b(v)=\prod_{n_k} [n_k]^2_q\in \mathbb{Z}[v]
   \bigg\}  \,.
\]
If we slightly relax the condition in the ring $\mathfrak{R}(y;\,v,t)$, we have
the following ring which is convenient in the $p$-adic argument in the following
subsection.
\[
  \mathcal{M}(y;\,q,t)
= \bigg\{ \frac{f(y;\,q,t)}{b(v)}:\, f\in
    \Omega (y)[q^{\pm 1/2},t^{\pm 1/2}],
    b(v)=\prod_{n_k} [n_k]^2_q\in \mathbb{Z}[v]\bigg\}\, .
\]
Given $\frac{f(y;\,q,t)}{b(v)}\in \mathcal{M}(y;\,q,t)$, if $f(y;\,q,t)$ is a
primitive polynomial in terms of $q^{\pm 1/2}$, $t^{\pm 1/2}$ and Schur
functions of $y$, we call $\frac{f(y;\,q,t)}{b(v)}$ is primitive.


\subsection{Multi-cover contribution and $p$-adic argument}

\begin{proposition}
\label{prop: T_d in L}$T_{\vec{d}}(y;\,q,t)\in \mathfrak{R}(y;v,t)$.
\end{proposition}

\begin{proof}
Recall the definition of quantum group invariants of links. By the formula of
universal $R$-matrix, it's easy to see that
\[
 W_{\vec{A}}(\mathcal{L})\in \mathcal{M}(y;\,q,t)\, .
\]
Since we have already proven the existence of the pole structure in LMOV
conjecture, Proposition \ref{prop: T_d in L} will be naturally satisfied if we
can prove
\begin{equation}\label{prop: T_d in ring M}
T_{\vec{d}}(y;\,q,t)\in \mathcal{M}(y;\,q,t)\, .
\end{equation}
Before diving into the proof of (\ref{prop: T_d in ring M}), let's do some
preparation. For
\[
 \forall\mathfrak{A}
=(\vec{A}_{1},\vec{A}_{2},...)\in \mathcal{P}(\mathcal{P}^{L}) \, ,
\]
define
\[
   \mathfrak{A}^{(d)}
=  \bigg( \underbrace{\vec{A}_{1},...,\vec{A}_{1}}_d,
    \underbrace{\vec{A}_{2},...,\vec{A}_{2}}_d,\ldots
   \bigg)\, .
\]

\begin{lemma}\label{lemma: theta multiplicity}
If $\theta_{\mathfrak{A}}=\frac{c}{d}$, $d>1 $ and $\gcd\left( c,d\right) =1$,
then for any $r|d$, we can find
$\mathfrak{B}\in\mathcal{P}(\mathcal{P}^{L})$ such that
$\mathfrak{A}=\mathfrak{B}^{(r)}$.
\end{lemma}

\begin{proof}
Let $\ell=\ell(\mathfrak{A})$, we have
\[
  \theta _{\mathfrak{A}}
= \frac{(-1)^{\ell(\mathfrak{A})-1}}{\ell(\mathfrak{A})}u_{%
    \mathfrak{A}}
= \frac{(-1)^{\ell(\mathfrak{A})-1} (\ell(\mathfrak{A})-1)!}
    {|\Aut \mathfrak{A}|}\,.
\]
Let
\[
  \mathfrak{A}
= \bigg(\underbrace{\vec{A}_{1},...,\vec{A}_{1}}_{m_1},
    \ldots,
    \underbrace{\vec{A}_{n},...,\vec{A}_{n}}_{m_{n}}
  \bigg)
    \, ,
\]
so $\ell(\mathfrak{A})=m_{1}+...+m_{n}$. Note that
\[
 u_{\mathfrak{A}}
=\binom{\ell}{m_{1},\,m_{2},\cdots ,\,m_{n}}.
\]
Let $\eta=\gcd (m_{1},\,m_{2},\cdots ,\,m_{n})$. We have
$\mathfrak{A}=\mathfrak{\widetilde{A}}^{(\eta)}$, where
\[
  \mathfrak{\widetilde{A}}
= \bigg(\underbrace{\vec{A}_{1},\ldots,\vec{A}_1}_{m_1/\eta},
    \ldots,
    \underbrace{\vec{A}_{n},\ldots,\vec{A}_{n}}_{m_n/\eta}
  \bigg) \, .
\]
By Corollary \ref{corollary:a devide (a: a-1, ..., a-n)} ,
$\frac{\ell}{\eta}|u_{\mathfrak{A}}$ and $\gcd (c,d)=1$, one has $d|\eta$. We
can take $\mathfrak{B}=\widetilde{\mathfrak{A}}^{(\frac{\eta}{r})}$. This
completes the proof.
\end{proof}

\begin{remark}
By the choice of $\eta$ in the above proof, we know $|\vec{A}_\alpha|$ is
divisible by $\eta$ for any $\alpha$.
\end{remark}

By Lemma \ref{lemma: theta multiplicity} and (\ref{eqn: integrality eqn}),
we know $T_{\vec{d}}$ is of form $f(y;\,q,t)/k$, where
$f\in \mathcal{M} (y;\,q,t) $, $k|\vec{d}$.
We will show $k$ is in fact $1$.

Given $\frac{r}{s}\frac{f(y;\,q,t)}{b(v)}$ where
\[
 \frac{f(y;\,q,t)}{b(v)}\in\mathcal{M}(y;\,q,t)
\]
is primitive, define
\begin{align} \label{def: Ord_p}
  \Ord_{p}\Big(\frac{r}{s} \frac{f( y;\,q,t)}{b(v)}\Big)
= \Ord_{p}\Big(\frac{r}{s}\Big) \, .
\end{align}

\begin{lemma}
Given $\mathfrak{A}\mathcal{\in P}\left( \mathcal{P}^{L}\right) $, $p$ any
prime number and $f_\mathfrak{A} (y;\,q,t) \in \mathcal{M} (y;\,q,t)$, we have
\[
  \Ord_p
    \Big(
    \theta _{\mathfrak{A}^{(p)}}f_{\mathfrak{A}^{(p)}}(y;\,q,t)
    -\frac{1}{p}\theta_{\mathfrak{A}}f_{\mathfrak{A}}(y^p;\,q^p,t^p)
    \Big)
\geq 0 \,.
\]
\end{lemma}

\begin{proof}
Assume $\theta_\mathfrak{A}=\frac{b}{^{p^{r}\cdot a}}$, where
$\gcd (p^{r}\cdot a,b)=1$, $p\nmid a$. In (\ref{eqn: theta minus sign}), minus
is taken except for one case: $p=2$ and $r=0$, in which the calculation is very
simple and the same result holds. Therefore, we only show the general case which
minus sign is taken. By Lemma \ref{lemma: theta multiplicity}, one can choose
$\mathfrak{B}\in \mathcal{P(P}^{L})$ such that
$\mathfrak{A}=\mathfrak{B}^{(p^{r})}$. Note that
$f_{\mathfrak{A}}=f_{\mathfrak{B}}^{p^{r}}$.

Let $s=\ell(\mathfrak{B})$ and
\[
 \mathfrak{B}
=\bigg(\underbrace{\vec{B}_{1},\ldots,\vec{B}_1}_{s_1},
    \ldots,
    \underbrace{\vec{B}_{k},\ldots,\vec{B}_{n}}_{s_k}
  \bigg).
\]
Since
$\ell(\mathfrak{A}^{(p)})=p\cdot \ell(\mathfrak{A})$,
by Theorem \ref{theorem: number theory},
\begin{align}
&\Ord_{p}
    \Big( \theta _{\mathfrak{A}^{(p)}}-\frac{1}{p}\theta _{\mathfrak{A}}
    \Big) \nonumber
\\
& = \Ord_{p}\frac{1}{p\cdot \ell(\mathfrak{A})}
    \bigg[ \binom{p^{r+1}s}{p^{r+1}s_{1},\cdots ,p^{r+1}s_{k}}
    -\binom{p^{r}s}{p^{r}s_{1},\cdots,p^{r}s_{k}}\bigg]
    \label{eqn: theta minus sign}
\\
& \geq 2(r+1)-(1+r) \nonumber
\\
& > 0 \label{ineq: multi-cover in theta}\,.
\end{align}
By Theorem \ref{thm: Fermat theorem},
\begin{align*}
& \Ord_p
    \Big[ \frac{1}{p}\theta_{\mathfrak{A}}
    \Big(f_{\mathfrak{A}^{(p)}}(y;\,q,t)
    -f_{\mathfrak{A}}(y^{p};\,q^{p},t^{p} \Big) \Big]
\\
& = \Ord_{p}
    \bigg[ \frac{b}{p^{r+1}a}
        \Big( f_{\mathfrak{B}}(y;\,q,t)^{p^{r+1}}
        -f_{\mathfrak{B}}(y^{p};\,q^{p},t^{p})^{p^{r}}\Big)
    \bigg]
\\
& \geq 0 \,.
\end{align*}

Apply the above two inequalities to
\begin{align*}
& \Ord_p
    \Big( \theta_{\mathfrak{A}^{(p)}}f_{\mathfrak{A}^{(p)}}(y;\,q,t)
    -\frac{1}{p}\theta_{\mathfrak{A}}f_{\mathfrak{A}}(y^{p};\,q^{p},t^{p})
    \Big)
\\
& = \Ord_p
    \Big[ \Big(\theta_{\mathfrak{A}^{(p)}}-\frac{1}{p}\theta_{\mathfrak{A}}\Big)
        f_{\mathfrak{A}^{(p)}}(y;\,q,t)
    +\frac{1}{p}\theta_{\mathfrak{A}}
    \Big(f_{\mathfrak{A}^{(p)}}(y;\,q,t)
    -f_{\mathfrak{A}}(y^{p};\,q^{p},t^{p} \Big) \Big]\,.
\end{align*}
The proof is completed.
\end{proof}

Apply the above Lemma, we have
\begin{equation}\label{eqn: calculation for lemma of multicover}
  \Ord_{p}\Big( \theta _{\mathfrak{A}^{(p)}}W_{\mathfrak{A}^{(p)}}(q,t)s_{%
    \mathfrak{A}^{(p)}}(z)-\frac{1}{p}\theta _{\mathfrak{A}}W_{\mathfrak{A}%
    }(q^{p},t^{p})s_{\mathfrak{A}}(z^{p})\Big)
\geq 0 \,.
\end{equation}

Let
\[
 \Phi_{\vec{d}}\left( y;\,q,t\right)
=\sum_{\mathfrak{A}\in \mathcal{P(P}^{n}),\,
 \parallel\mathfrak{A}\parallel =\vec{d}}\theta _{\mathfrak{A}}
 W_{\mathfrak{A}}(q,t)s_{\mathfrak{A}}(z) \, .
\]
By Lemma \ref{lemma: theta multiplicity}, one has
\[
 \Big\{ \mathfrak{B}:\; \parallel \mathfrak{B} \parallel = p\,\vec{d}
    \textrm{ and } \Ord_p(\theta_{\mathfrak{B}})<0
 \Big\}
=\Big\{ \mathfrak{A}^{(p)}:\; \parallel \mathfrak{A} \parallel = \vec{d}\,
 \Big\}
\]
Therefore, By (\ref{eqn: calculation for lemma of multicover}),
\begin{align*}
& \Ord_p \big( \Phi_{p\vec{d}}\,(y;\,q,t )
 -\frac{1}{p} \Phi_{\vec{d}}\, ( y^{p};q^{p},t^{p}) \big)
\\
&=\sum_{\parallel\mathfrak{A}\parallel =\vec{d}}
 \Ord_{p}\Big( \theta _{\mathfrak{A}^{(p)}}W_{\mathfrak{A}^{(p)}}(q,t)s_{%
    \mathfrak{A}^{(p)}}(z)-\frac{1}{p}\theta _{\mathfrak{A}}W_{\mathfrak{A}%
    }(q^{p},t^{p})s_{\mathfrak{A}}(z^{p})\Big)
\\
&\geq 0 \,.
\end{align*}
We have thus proven the following Lemma.

\begin{lemma}\label{lemma: multi-cover contribution}
For any prime number $p$ and $\vec{d}$,
\[
 \Ord_{p}
 \Big( \Phi_{p\vec{d}}\,(y;\,q,t )
 -\frac{1}{p} \Phi_{\vec{d}}\, ( y^{p};q^{p},t^{p})
 \Big)
\geq 0 \,.
\]
\end{lemma}

For any $p|\vec{d}$, by (\ref{eqn: integrality eqn}),
\begin{align*}
    T_{\vec{d}}
&= q^{|\vec{d}|} \sum_{k|\vec{d}} \frac{\mu(k)}{k}
    \sum_{
      \mathfrak{A}\in\mathcal{P(P}^L),\,
      \parallel\mathfrak{A}\parallel =\vec{d}/k
     }
    \theta_{\mathfrak{A}} W_{\mathfrak{A}}(q^{k},t^{k})
    s_{\mathfrak{A}}(z^{k})
\\
&= q^{|\vec{d}|}
    \bigg(
      \sum_{k|\vec{d},\, p\nmid k}
    \frac{\mu(k)}{k} \Phi_{\vec{d}/k}(y^{n};\,q^{n},t^{n})
      +\sum_{k|\vec{d},\, p\nmid k}
    \frac{\mu(pk)}{pk} \Phi_{\vec{d}/(pk)}(y^{pk};\,q^{pk},t^{pk})
    \bigg)
\\
&= q^{|\vec{d}|}
    \sum_{k|\vec{d},\, p\nmid k} \frac{\mu(k)}{k}
    \Big(
        \Phi_{\vec{d}/k}(y^{k};\,q^{k},t^{k})
    -\frac{1}{p}\Phi _{\vec{d}/(pk)}(y^{pk};\,q^{pk},t^{pk})
    \Big) \,.
\end{align*}
By Lemma \ref{lemma: multi-cover contribution},
\begin{equation}\label{ineq: Ord_p T_d geq 0}
 \Ord_{p}T_{\vec{d}}\geq 0\,.
\end{equation}
By the arbitrary choice of $p$, we prove that
$T_{\vec{d}}\in\mathcal{M}(y;\,q,t)$,
hence
\[
 T_{\vec{d}}\in \mathfrak{R}(y;v,t)\,.
\]
The proof of Proposition \ref{prop: T_d in L} is completed.
\end{proof}


\subsection{Integrality}

    By \eqref{eqn: P_B of F_mu} and Propositions \ref{prop: pole structure of tilde_F_mu},
    (note that $\phi_{\vec{\mu}/d}(q^d,t^d) = \phi_{\vec{\mu}}(q,t)$.)
    \begin{align*}
        P_{\vec{B}}(q,t)
            &=\sum_{\vec{\mu}} \frac{\chi_{\vec{B}}(\vec{\mu})}{\phi_{\vec{\mu}}(q)}
                \sum_{d|\vec{\mu}} \frac{\pi (d)}{d}F_{\vec{\mu}/d}(q^{d},t^{d}) \\
            &=\sum_{ \vec{\mu} } \chi_{\vec{B}} ( \vec{\mu} ) \sum_{d| \vec{\mu} }
                \frac{ \mu(d) }{d} \widetilde{F}_{\vec{\mu}/d} (q^d, t^d) \\
            &=\sum_{ \vec{\mu} } \chi_{\vec{B}} ( \vec{\mu} ) \sum_{d| D_{\vec{\mu}} }\frac{ \mu(d) }{d}
                \frac{ H_{\vec{\mu}/D_{\vec{\mu}} }(t^{D_{\vec{\mu}}}) }{ [D_{\vec{\mu}}]^2 } + \textrm{polynomial} \\
            &=\sum_{ \vec{\mu} } \chi_{\vec{B}} ( \vec{\mu} ) \delta_{1, D_{\vec{\mu}}}
                \frac{ H_{\vec{\mu}/D_{\vec{\mu}} }(t^{D_{\vec{\mu}}}) }{ [D_{\vec{\mu}}]^2 } + \textrm{polynomial}\,,
    \end{align*}
    where $\delta_{1,n}$ equals $1$ if $n=1$ and $0$ otherwise.
    It implies that $P_{\vec{B}}$ is a rational function which only has pole at $q=1$.
    In the above computation, we used a fact of M\"obius inversion,
    \begin{align*}
        \sum_{d|n} \frac{ \mu(d) }{d} = \delta_{1,n}\,.
    \end{align*}

Therefore, for each $\vec{B}$,
\[
 \sum_{g= 0}^\infty \sum_{Q\in \mathbb{Z}/2}
    N_{\vec{B};\,g,Q}(q^{-1/2}-q^{1/2})^{2g}t^Q
\in \mathbb{Q}[(q^{-1/2}-q^{1/2})^2, t^{\pm \frac{1}{2}}]\,.
\]

On the other hand, by Proposition \ref{prop: T_d in L},
$T_{\vec{d}}\in \mathfrak{R}(y;\,v,t)$ and
$\Ord_p T_{\vec{d}} \geq0$ for any
prime number $p$. We have
\[
 \sum_{|\vec{B}|=\vec{d}} N_{\vec{B};\,g,Q} s_{\vec{B}} (y)
\in \Omega(y) \,.
\]
This implies $N_{\vec{B};\,g,Q}\in\mathbb{Z}$.

Combine the above discussions, we have
\[
    \sum_{g= 0}^\infty \sum_{Q\in \mathbb{Z}/2}
    N_{\vec{B};\,g,Q}(q^{-1/2}-q^{1/2})^{2g}t^{Q}
\in \mathbb{Z}[(q^{-1/2}-q^{1/2})^{2},t^{\pm 1/2}] \, .
\]
The proof of Theorem \ref{thm: integrality} is completed.

\section{Concluding Remarks and Future Research}

In this section, we briefly discuss some interesting problems related to string
duality which may be approached through the techniques developed in this paper.

\subsection{Duality from a mathematical point of view}

Let
$
 \mathbf{p}
=(\mathbf{p}^1, \ldots, \mathbf{p}^L),
$
where
$
 \mathbf{p}^\alpha
=(p^\alpha_1,p^\alpha_2, \ldots,) \,.
$
Defined the following generating series of open Gromov-Witten invariants
\[
 F_{g,\vec{\mu}}(t,\tau)
=\sum_{\beta} K^\beta_{g,\vec{\mu}}(\tau) e^{\int_\beta \omega}
\]
where $\omega$ is the K\"ahler class of the resolved conifold,
$\tau$ is the framing parameter and
\begin{align*}
 t=e^{\int_{\mathbb{P}^1} \omega}, \quad \textrm{and} \quad
    e^{\int_\beta \omega} = t^Q\,.
\end{align*}
Consider the following generating function
\[
 F(\mathbf{p};\,u,t;\,\tau)
=\sum_{g=0}^\infty \sum_{\vec{\mu}} u^{2g-2+\ell(\vec{\mu})}
    F_{g,\vec{\mu}}(t;\, \tau)
    \prod_{\alpha=1}^L p^\alpha_{\mu^\alpha} \,.
\]
It satisfies the log cut-and-join equation
\[
 \frac{\partial F(\mathbf{p};\,u,t;\,\tau)}{\partial \tau}
=\frac{u}{2} \sum_{\alpha=1}^L \mathfrak{L}_\alpha
    F(\mathbf{p};\,u,t;\,\tau)\,.
\]
Therefore, duality between Chern-Simons theory and open Gromov-Witten theory
reduces to verifying the uniqueness of the solution of cut-and-join equation.

Cut-and-join equation for Gromov-Witten side comes from the degeneracy and
gluing procedure while uniqueness of cut-and-join system should in principle be
obtained from the verification at some initial value. However, it seems very
difficult to find a suitable initial value. For example, in the case of
topological vertex theory, cut-and-join system has singularities when the
framing parameter takes value at $0$, $-1$, $\infty$ while these points are the
possible ones to evaluate at. One solution might be through studying
Riemann-Hilbert problem on controlling the monodromy at three singularity
points. When the case goes beyond, the situation will be even more complicated.
A universal method of handling uniqueness is required for the final proof of
Chern-Simons/topological string duality conjecture.

A new hope might be found in our development of cut-and-join analysis. In the
log cut-and-join equation \eqref{eqn: non-linear cut-and-join}, the non-linear
terms reveals the important recursion structure. For the uniqueness of
cut-and-join equation, it will appear as the vanishing of all non-linear terms.
We will put this in our future research.

\subsection{Other related problems}

There are many other problems related to our work on LMOV conjecture. We
briefly list some problems that we are working on.

Volume conjecture was proposed by Kashaev in \cite{Kashaev} and
reformulated by \cite{Murakami-Murakami}. It relates the volume of
hyperbolic 3-manifolds to the limits of quantum invariants. This
conjecture was later generalized to complex case \cite{MMOTY} and to
incomplete hyperbolic structures \cite{Gukov}. The study of this
conjecture is still staying at a rather primitive stage
\cite{Murakami-Yokota,Kashaev-Tirkkonen, Zheng,Garoufalidis-Le}.

LMOV conjecture has shed new light on volume conjecture. The cut-and-join
analysis we developed in this paper combined with rank-level duality in
Chern-Simons theory seems to provide a new way to prove the existence of the
limits of quantum invariants.

There are also other open problems related to LMOV conjecture. For example,
quantum group invariants satisfy skein relation which must have some
implications on topological string side as mentioned in \cite{LM}. One could
also rephrase a lot of unanswered questions in knot theory in terms of open
Gromov-Witten theory. We hope that the relation between knot theory and open
Gromov-Witten theory will be explored much more in detail in the future. This
will definitely open many new avenues for future research.

\pagebreak

\appendix

\section{Appendix}

\subsection{}
\label{appendix: T_vd calculation}

Here we carry out the calculation in Section
\ref{subsec: Integrality - ring structure}.
\begin{align*}
 T_{\vec{d}}
&= \sum_{| \vec{B}| =\vec{d}}s_{\vec{B}}(y)
    P_{\vec{B}}(q,t)
\\
&= \sum_{| \vec{B}| =\vec{d}}\frac{s_{\vec{B}}(y)
    \chi _{%
\vec{B}}(\vec{\mu})}{\phi _{\vec{\mu}}(q)}
    \sum_{k|\vec{\mu}}\frac{\mu (k)}{k}
    F_{\vec{\mu}/k}\left( q^{k},t^{k}\right)
\\
&= \sum_{k|\vec{\mu}\, ,|\vec{\mu}| =\vec{d}}
    \frac{\mu (k)}{k}\frac{p_{\vec{\mu}}(y)}{\phi _{\vec{\mu}}(q)}
    \sum_{\Lambda \in \mathcal{P(P}^{n}),\,|\Lambda| =\vec{\mu}/k}
    \theta_{\Lambda}Z_{\Lambda }(q^{k},t^{k})
\\
&= \sum_{k|\vec{\mu}\, ,|\vec{\mu}| =\vec{d}}
    \frac{\mu (k)}{k}
    \sum_{\parallel\Lambda \parallel =\vec{\mu}/k}
    \theta_{\Lambda}\phi^{-1}_\Lambda(q^k)p_\Lambda(y^k)
    Z_{\Lambda }(q^{k},t^{k})
\\
&= \sum_{k|\vec{\mu}\, ,|\vec{\mu}| =\vec{d}}
    \frac{\mu(k)}{k}
    \sum_{\parallel\Lambda \parallel =\vec{\mu}/k}
    \theta_{\Lambda}\phi^{-1}_\Lambda(q^k)p_\Lambda(y^k)
    \prod_{\beta=1}^{\ell(\Lambda)} \sum_{\vec{\mathfrak{A}}_\beta}
    \frac{\chi_{\vec{\mathfrak{A}}_\beta}(\vec{\Lambda}_\beta)}
        {\mathfrak{z}_{\vec{\Lambda}_\beta}}
    W_{\vec{\mathfrak{A}}_\beta}(q^k)
\\
&= \sum_{k|\vec{d}} \frac{\mu(k)}{k}
    \sum_{|\vec{\mathfrak{A}}_\beta|=|\vec{\Lambda}_\beta|=d_\beta/k }
    \frac{(-1)^{\ell(\Lambda)-1}}{\ell(\Lambda)} \cdot \fu_\Lambda
  \\
  &  \qquad \qquad \times
    \prod_{\beta=1}^{\ell(\Lambda)}
    \bigg\{
    \phi^{-1}_{\vec{\Lambda}_\beta}(q^k)
        p_{\vec{\Lambda}_\beta}(y^k)
    \sum_{\vec{\mathfrak{A}}_\beta}
    \frac{\chi_{\vec{\mathfrak{A}}_\beta}(\vec{\Lambda}_\beta)}
    {\mathfrak{z}_{\vec{\Lambda} _\beta}}
    W_{\vec{\mathfrak{A}}_\beta}(q^k,t^k) \bigg\}
\\
&= \sum_{k|\vec{\mu}\, ,|\vec{\mu}| =\vec{d}}
    \frac{\mu(k)}{k}
    \sum_{\mathfrak{A}=(\vec{\mathfrak{A}}_1,\ldots),\,
          \parallel\mathfrak{A}\parallel=\vec{d}/k
         }
    \frac{(-1)^{\ell(\mathfrak{A})-1}}{\ell(\mathfrak{A})}
        \cdot \fu_{\mathfrak{A}}
  \\
  & \qquad \qquad \times
    \prod_{\beta=1}^{\ell(\mathfrak{A})}
    \bigg\{
    W_{\vec{\mathfrak{A}}_\beta}(q^k,t^k)
    \sum_{\vec{\Lambda}_\beta}
    \frac{\chi_{\vec{\mathfrak{A}}_\beta}(\vec{\Lambda}_\beta)}
        {\mathfrak{z}_{\vec{\Lambda} _\beta}}
    \phi^{-1}_{\vec{\Lambda}_\beta}(q^k)p_{\vec{\Lambda}_\beta}(y^k)
    \bigg\}
\\
&= q^{|\vec{d}|} \sum_{k|\vec{d}} \frac{\mu(k)}{k}
    \sum_{\mathfrak{A}\in\mathcal{P}(\mathcal{P}^L),\,
    \parallel\mathfrak{A}\parallel =\vec{d}/k}
    \theta_{\mathfrak{A}} W_{\mathfrak{A}}(q^k,t^k) s_{\mathfrak{A}}(z^k)
\end{align*}
where
\[
 p_n(z)
=p_n(y)\cdot p_n(x_i=q^{i-1})\, .
\]

\subsection{}

\begin{lemma}
\label{lemma: p-r factor of n=2} $p$ is prime, $r\geq 1$. Then
\[
 \binom{p^{r}a}{p^{r}b}-\binom{p^{r-1}a}{p^{r-1}b}
\equiv 0\mod(p^{2r})
\]
\end{lemma}

\begin{proof}
Consider the ratio
\begin{align*}
    \frac{\binom{p^{r}a}{p^{r}b}}{\binom{p^{r-1}a}{p^{r-1}b}}
&= \frac{\prod_{k=1}^{p^{r}b}\frac{(a-b)p^{r}+k}{k}}
    {\prod_{k=1}^{p^{r-1}b}\frac{(a-b)p^{r-1}+k}{k}}
\\
&= \prod_{ \substack{\gcd (k,p)=1, \\ 1\leq k\leq p^{r}b } }
    \frac{(a-b)p^{r}+k}{k}
\\
&\equiv 1+p^{r}(a-b)
    \sum_{ \substack{\gcd (k,p)=1, \\ 1\leq k\leq p^{r}b }}
    \frac{1}{k}\ \mod(p^{2r}) \,.
\end{align*}
Let
\[
A_{p}(n)=\sum_{\substack{1\leq k\leq n,\\ \gcd (k,p)=1} }\frac{1}{k} \,.
\]
Therefore, the proof of the lemma can be completed by showing
\[
 A_{p}(p^{r}b)=\frac{p^{r}c}{d}
\]
for some $c,d$ such that $\gcd(d,p)=1$.

If $\gcd(k, p)=1$, there exist $\alpha_k$, $\beta_k$ such that
\[
 \alpha_k k + \beta_k p^r =1 \,.
\]
Let
\[
 B_p(n)=\sum_{1\leq k \leq n, \gcd(k, p)=1}k\,.
\]
By the above formula,
\begin{align*}
  A_p(p^rb)
&\equiv b A_p(p^r) \mod(p^r)
\\
&\equiv b B_p(p^r) \mod(p^r) \,,
\end{align*}
and
\begin{align*}
    B_p(p^r)
&= \sum_{k=1}^{p^r}k-p\sum_{k=1}^{p^{r-1}}k
\\
&= \frac{p^r(p^r+1)}{2}-p\frac{p^{r-1}(p^{r-1}+1)}{2}
\\
&= \frac{p^{2r-1}(p-1)}{2} \,.
\end{align*}
Here, we have $2r-1\geq r$ since $r\geq 1$.
The proof is then completed.
\end{proof}

The following theorem is a simple generalization of the above Lemma.
\begin{theorem}\label{theorem: number theory}
$\sum_{i=1}^{n}a_{i}=a$, $p$ is prime,
$r\geq 1$, then
\[
  \binom{p^{r}a}{p^{r}a_{1},\cdots ,p^{r}a_{n}}
 -\binom{p^{r}a}{p^{r-1}a_{1},\cdots ,p^{r-1}a_{n}}
\equiv 0\mod(p^{2r})
\]
\end{theorem}

\begin{proof}
We have
\begin{align*}
& \binom{p^{r}a}{p^{r}a_{1},\cdots ,p^{r}a_{n}}
    -\binom{p^{r}a}{p^{r-1}a_{1},\cdots ,p^{r-1}a_{n}}
\\
& \qquad \qquad \qquad
   =\prod_{k=1}^{n}\binom{p^{r}(a-\sum_{i=1}^{k-1}a_{i})}{p^{r}a_{k}}
    -\prod_{k=1}^{n}\binom{p^{r-1}(a-\sum_{i=1}^{k-1}a_{i})}{p^{r-1}a_{k}}
  \qquad
\\
& \qquad \qquad \qquad \equiv 0\mod(p^{2r}) \,.
\end{align*}
In the last step, we used the Lemma \ref{lemma: p-r factor of n=2}. The proof is
completed.
\end{proof}

\begin{lemma}\label{lemma:a devide combination (a,b)}
We have
\[
 \frac{a}{\gcd (a,b)} \left| \binom{a}{b}\right. \,.
\]
\end{lemma}

\begin{proof}
Notice that
\[
\binom{a}{b}=\frac{a}{b}\binom{a-1}{b-1}\, ,
\]
i.e.
\[
 \frac{b}{\gcd (a,b)}\binom{a}{b}=\frac{a}{\gcd (a,b)}\binom{a-1}{b-1}\, .
\]
However,
\[
\gcd \left( \frac{a}{\gcd (a,b)},\frac{b}{\gcd (a,b)}\right) =1\, ,
\]
so
\[
 \frac{a}{\gcd (a,b)}\left|\binom{a}{b}\right.\, .
\]
\end{proof}

This direct leads to the following corollary.

\begin{corollary}
\label{corollary:a devide (a: a-1, ..., a-n)} $a=a_{1}+a_{2}+\cdots +a_{n}$.
Then
\[
\left.\frac{a}{\gcd (a_{1},a_{2},\cdots ,a_{n})}
\right| \binom{a}{a_{1},\,a_{2},\cdots ,\,a_{n}}\, .
\]
\end{corollary}

\begin{lemma}
 $a,r\in \mathbb{N}$, $p$ is a prime number, then
\[
 a^{p^r}-a^{p^{r-1}} \equiv 0 \mod(p^r)\,.
\]
\end{lemma}

\begin{proof}
 If $a$ is $p$, since $p^{r-1}\geq r$, the claim is true. If $\gcd(a,p)=1$, by
Fermat theorem, $a^{p-1}\equiv 1\mod(p)$. We have
\begin{align*}
  a^{p^r}-a^{p^{r-1}}
&= a^{p^{r-1}} \Big( (kp+1)^{p^{r-1}(p-1)}-1 \Big)
\\
&= a^{p^{r-1}} \sum_{i=1}^{p^{r-1}} (kp)^i \binom{p^{r-1}}{i}
\\
&\equiv 0 \mod (p^r)\,.
\end{align*}
Here in the last step, we used Lemma \ref{lemma:a devide combination (a,b)}.
\end{proof}

A direct consequence of the above lemma is the following theorem.

\begin{theorem}\label{thm: Fermat theorem}
 Given $f(y;\,q,t) \in \mathcal{M}(y;\,q,t)$, we have
\[
 \Ord_p \Big( f(y;\,q,t)^{p^{r+1}}
 - f(y^p;\,q^p,t^p)^{p^r} \Big)
\geq r+1 \,.
\]
\end{theorem}

\pagebreak

\vspace{5pt}

\noindent
\textsc{Center of Mathematical Sciences \\
Zhejiang University, Box 310027 \\
Hangzhou, China}

\noindent
\textsc{Department of mathematics \\
University of California at Los Angeles \\
Box 951555 \\
Los Angeles, CA 90095-1555 \\}
Email: liu@math.ucla.edu.

\vspace{10pt}

\noindent
\textsc{Department of Mathematics \\
Harvard University \\
One Oxford Street \\
Cambridge, MA, 02138 \\ }
Email: ppeng@math.harvard.edu.

\end{document}